\documentclass[11pt]{article}
\usepackage[usenames,dvipsnames]{xcolor}
\usepackage{amsmath, amssymb,amsthm}
\usepackage[colorlinks=true,allcolors  = black,linkcolor = gray]{hyperref}             
\usepackage{amsthm}
\usepackage{pstricks-add,comment}

\usepackage{enumerate}

\newcommand*{\Prob}{\mathrm{\mathbf{P}}}

\newcommand*{\F}{\mathcal{F}}

\newcommand*{\cH}{\mathcal{H}}
\newcommand*{\cP}{\mathcal{P}}
\newcommand*{\cA}{\mathcal{A}}
\newcommand*{\cR}{\mathcal{R}}

\newcommand*{\K}{\mathcal{K}}
\newcommand*{\A}{\mathcal{A}}
\newcommand*{\cC}{\mathcal{C}}
\newcommand*{\cJ}{\mathcal{J}}

\newcommand{\bN}{\mathbf{N}}
\newcommand{\Cnk}[1]{ C_n^{(k)}(#1) }

\usepackage{accents}
\newcommand{\seq}[1]{\accentset{\rightharpoonup}{#1}}

\def \pc {properly colored~}

\newtheorem{theorem}{Theorem}
\newtheorem{lemma}{Lemma}

\newtheorem{prop}[lemma]{Proposition}

\let\eps=\varepsilon

            \newcommand{\nak}[1]{{#1}}

\newtheorem{conj}{Conjecture}
\newtheorem{claim}{Claim}
\newtheorem{cor}[lemma]{Corollary}

\begin{document}

\title{Properly colored Hamilton cycles in Dirac-type hypergraphs}

\author{ Sylwia Antoniuk
\\A. Mickiewicz University\\Pozna\'n, Poland\\{\tt antoniuk@amu.edu.pl}
\and Nina Kam\v cev
\thanks{Research supported by ARC Discovery Project DP180103684.}
\\Monash University\\Melbourne, Australia\\{\tt nina.kamcev@monash.edu}
\and Andrzej Ruci\'nski
\thanks{Research supported by Polish NSC grant 2018/29/B/ST1/00426.}
\\A. Mickiewicz University\\Pozna\'n, Poland\\{\tt rucinski@amu.edu.pl}}

\date{\today}

\maketitle

\let\thefootnote\relax
\let\thefootnote\relax

\begin{abstract}
    We consider a robust variant of Dirac-type problems in $k$-uniform hypergraphs. For instance, we prove that if $\cH$ is a $k$-uniform hypergraph with minimum codegree at least $\left(\frac 12 + \gamma \right)n$, $\gamma >0$, and $n$ sufficiently large, then any edge coloring $\phi$ satisfying appropriate local constraints yields a properly colored tight Hamilton cycle in $\cH$. Similar results for loose cycles are also shown.
\end{abstract}


\section{Introduction}\label{intro}
        Our study lies at the intersection of three classical areas of research -- we will extend several Dirac-type theorems for hypergraphs to an edge-colored setting. The famous theorem of Dirac states that any $n$-vertex graph with minimum degree $n/2$ has a Hamilton cycle. Among the numerous research lines stemming from the Theorem are so-called Dirac-type problems, where one aims to embed a spanning subgraph into a graph of given minimum degree.

        Dirac-type problems in hypergraphs, first considered by Katona and Kierstead~\cite{kk99} in 1999, now constitute a fruitful and dynamic theory. One reason  for this multitude of results is that there is a number of different notions of hypergraph degrees and cycles, and consequently several extremal constructions.
        Given $\ell \in [k-1]$, a \emph{$(k,\ell)$-overlapping cycle} $C_n^{(k)}(\ell)$, or shortly an \emph{$(k,\ell)$-cycle}  is an $n$-vertex $k$-graph whose vertices can be ordered cyclically such that each of its edges consists of $k$ consecutive vertices and every two consecutive edges share exactly $\ell$ vertices. Note that $C_n^{(k)}(\ell)$ exists whenever $k- \ell$ divides $n$ (we will further use standard notation $(k-\ell)|n$ to indicate this fact). For $k=2$ and $\ell=1$, this reduces to the graph cycle. The two extreme cases, $\ell=1$ and $\ell=k-1$, are usually referred to, respectively, as \emph{loose} and \emph{tight} cycles. We say that a $k$-graph contains a \emph{Hamilton $(k,\ell)$-cycle}, or is \emph{$\ell$-Hamiltonian} if
it contains an  $(k,\ell)$-cycle as a spanning subhypergraph.

    The \emph{degree} of a set $S \subset V(\cH)$ in a hypergraph $\cH$, denoted by $\deg(S)$, is the number of edges of $\cH$ containing $S$.
        For sets of order 1 or $k-1$, the terms \emph{vertex degree} and \emph{codegree} (respectively) are often used. \emph{The minimum and maximum $s$-degree} of $\cH$ are defined as
    $$\delta_{s}(\cH) = \min \{\deg(S): S \subset V(\cH), |S| = s \},      \  \Delta_{s}(\cH) = \max \{\deg(S): S \subset V(\cH), |S| = s \}.$$
    The following three results are directly relevant to our topic.
        \begin{theorem}\label{known}  Let $1\le \ell<k/2$ and $\gamma >0$. Let $\cH$ be a $k$-graph on $n$-vertices for sufficiently large $n$.
        \begin{enumerate}
            \item \emph{~\cite{rrs08}} If $\delta_{k-1}(\cH) \geq \left(\frac 12 + \gamma \right) n$, then $\cH$ contains a tight Hamilton cycle.
            \item \emph{~\cite{kmo10,hs10}} If $\delta_{k-1}(\cH) \geq \left(\frac {1}{2(k-\ell)} + \gamma \right) n$, then $\cH$ contains a Hamilton $(k,\ell)$-cycle.
            \item \emph{~\cite{bhs13}} If $k=3$ and $\delta_{1}(\cH) \geq \left( \frac{7}{16} + \gamma \right)n^2/2$, then $\cH$ contains a loose Hamilton cycle.
        \end{enumerate}
        \end{theorem}
         The minimum-degree conditions in all three statements are asymptotically tight. In some cases, the corresponding exact thresholds are known~\cite{rrs11,hz15vert,hz15}. The most recent survey covering hypergraph Dirac-type problems is~\cite{zhao16}. The area has witnessed the development of several ubiquitous techniques such as the absorbing method and tools in hypergraph regularity, including the hypergraph blow-up lemma~\cite{keevash11}. As we shall see in our proofs, hypergraph Tur\'an problems also play an important role.

        In our setting, the host graph will be edge-colored and, given appropriate restrictions on the coloring, we will be embedding properly colored Hamilton cycles. A hypergraph is \emph{\pc}if every two intersecting edges are assigned different colors. Our results can be seen as \emph{robust} versions of the above-mentioned Dirac-type theorems.

        Sufficient conditions for finding properly colored Hamilton cycles in edge-colorings of the complete graph $K_n$ were first proposed by Daykin~\cite{daykin76} in 1976. Let us call an edge colouring of a hypergraph $\cH$ \emph{locally $t$-bounded} if any colour appears on at most $t$ edges incident to any given vertex. Daykin conjectured that there is a positive constant $\mu$ such that for any $n$, a locally $\mu n$-bounded colouring of $K_n$ contains a properly colored Hamilton cycle. The conjecture was resolved by Chen and Daykin~\cite{chen}, as well as Bollob\'as and Erd\"os~\cite{be76}. After a number of improvements and extensions (see, e.g.,~\cite{shearer,alongutin,bkp12}), Lo~\cite{lo16} showed that Daykin's conjecture holds for any $\mu < \frac 12$ and large $n$, which is asymptotically optimal. Analogous `global' sufficient conditions for rainbow embeddings  (in which no two edges have the same colour) have often been proved in parallel, also motivated by famous problems which can be encoded in terms of rainbow subgraphs, such as Latin transversals. However, we will streamline most of our discussion towards the properly colored setting.

    The question of Daykin also has natural hypergraph analogues. Dudek, Frieze and Ruci\'nski \cite{dfr12} showed that for $\mu >0$, $n$ sufficiently large and $\ell \in [k-1]$, locally $\mu n^{k-\ell}$-bounded colourings of the complete $k$-uniform hypergraph $K_n^{(k)}$ contain a properly coloured $\ell$-overlapping cycle. However, this is only known to be tight for the loose cycle $\Cnk{1}$, and the problem remains wide open for tight cycles. In~\cite{df15} and ~\cite{ksv17}, the boundedness condition was weakened to codegrees -- in particular, if the subhypergraph induced by any color has maximum codegree at most $\mu n$, then one can find a tight Hamilton cycle.

        Several authors have considered replacing the host (hyper)graph by an incomplete (hyper)graph~\cite{kls17,kls16,gj18,cps17,cp18,ckpy18}.  Krivelevich, Lee and Sudakov have shown that Daykin's conjecture holds even when $K_n$ is replaced by an arbitrary graph of minimum degree at least $n/2$ (often called a \emph{Dirac graph}), thus confirming a conjecture of H\"aggvist. They placed their result in the context of a number of recent studies on \emph{robustness} of extremal and probabilistic results~\cite{sudakov17}.

        We will show that a similar phenomenon occurs in $k$-uniform hypergraphs. A \emph{colored hypergraph} is a pair $(\cH,\phi)$, where $\cH$ is  a $k$-graph and $\phi: \cH\to \bN$ is a coloring of the edges of $\cH$. For each $i\in\bN$, we denote by $\cH_i=\{e\in \cH: \phi(e)=i\}$ the subhypergraph of $\cH$ consisting of the edges of  color $i$.
We now state our three new results which correspond to the three parts of Theorem \ref{known} above. In each of them we make a suitable assumption on the coloring in terms of $\Delta_\ell(\cH_i)$.

\begin{theorem}\label{tight} For every  $k\ge3$ and $\gamma>0$ there exist $c>0$ and $n_0>0$ such that if  $(\cH,\phi)$ is an $n$-vertex  colored $k$-graph with $n \geq n_0$, $\delta_{k-1}(\cH)\ge(1/2+\gamma)n$ and $\Delta_{k-1}(\cH_i)\le cn$ for every~$i\in\bN$, then $(\cH,\phi)$ contains a properly colored tight Hamilton cycle $C_n^{(k)}(k-1)$.
\end{theorem}

\begin{theorem}\label{ell-cycles}
For every $k\ge3$, $1\le\ell<k/2$ and $\gamma>0$ there exist $c>0$ and $n_0>0$ such that if  $(\cH,\phi)$ is an $n$-vertex colored $k$-graph with $(k-\ell)\,|\,n\ge n_0$, \newline $\delta_{k-1}(\cH)\ge\left(\tfrac1{2(k-\ell)}+\gamma\right)n$ and  $\Delta_{\ell}(\cH_i)\le cn^{k-\ell}$ for every~$i\in\bN$, then $(\cH,\phi)$ contains a properly colored Hamilton $(k,\ell)$-cycle $\Cnk{\ell}$.
\end{theorem}


\begin{theorem}\label{loose_singles}
For every $\gamma>0$ there exist $c>0$ and $n_0>0$ such that
if  $(\cH,\phi)$ is an $n$-vertex colored $3$-graph with $2\,|\,n\ge n_0$,
 $\delta_1(\cH)\ge(7/16+\gamma)n^2/2$ and $\Delta_1(\cH_i)\le cn^2$ for every~$i\in\bN$, then $(\cH,\phi)$ contains a properly colored loose Hamilton cycle $C_n^{(3)}(1)$.
\end{theorem}
As, trivially,  $\Delta_{\ell}(\cH_i)\le\binom{n-\ell}{k-\ell}$, the assumptions on $\Delta_{\ell}(\cH_i)$ in our theorems are optimal up to a constant factor.

     In our proofs we adapt the absorbing method to the properly colored setting, which may turn out to be a fruitful direction given that alternative approaches do not extend straightforwardly. For instance,  a uniformly random vertex ordering analysed using the Local Lemma (used, for instance, in~\cite{bkp12,ksv17}), immediately fails for host hypergraphs which are not `almost' complete. Recently, the blow-up lemma for graphs has been adapted and fruitfully applied to the rainbow setting~\cite{gj18}, but the hypergraph blow-up lemma used in~\cite{kkmo11} is significantly more intricate.

      Finally, Theorems~\ref{tight}-\ref{loose_singles} provide new evidence in support of a meta-conjecture formulated by Coulson, Keevash, Perarnau and Yepremyan~\cite{ckpy18} in the context of rainbow problems. Namely, for a large class of Dirac-type problems for $k$-graphs, the rainbow counterparts for bounded colorings should have asymptotically the same \emph{degree threshold} as the original problem. Our results show that the  conjecture holds for Hamilton cycles if the rainbow colorings are replaced by less restrictive proper colorings.
       To our knowledge, these are the first results of this type on embedding spanning hypergraphs.
	
    The paper is structured as follows. Section~\ref{1} contains preliminary definitions and tools. Theorems~\ref{tight}, \ref{ell-cycles} and \ref{loose_singles} are proved in Sections~\ref{sec:tight}, \ref{sec:ell-cycle} and \ref{sec:loose} respectively.



\section{Preliminaries}  \label{1}

 In this preliminary section we give some definitions and results which are relevant throughout the paper.

 We begin with two simple existential results whose proofs employ the standard probabilistic method. The first of them is cited from \cite{rrss17} (see Lemma 3.10 therein). We will use it three times when proving respective 'reservoir lemmas'. However, part (c) will be needed only in Section~\ref{con-res-loose}.

\begin{prop}[\cite{rrss17}]\label{res} For every $p$, $0<p<1$, there is $n_0$ such that the following holds. Let $U_1,\dots,U_s$ be subsets of an $n$-element set $V$, $n\ge n_0$, and let $G_1,\dots, G_t$ be graphs on $V$, where  $s$ and $t$ are both polynomials in $n$. Further, let $|U_i|\ge \alpha_i n$, $i=1,\dots,s$, and $|G_j|\ge \beta_j \binom{n}2$, $j=1,\dots,t$, for some constants $0<\alpha_i,\beta_j<1$. Then there exists a subset $R\subset V$ such that
\begin{enumerate}
\item[(a)] $\big||R|-pn\big|\le pn^{2/3}$,

\item[(b)] for all $i=1,\dots,s$, we have $|U_i\cap R|\ge (\alpha_i-2n^{-1/3})|R|$, and

\item[(c)] for all $i=1,\dots,t$, we have $|G_j[R]|\ge (\beta_j-3n^{-1/3})\binom{|R|}2$. \qed
\end{enumerate}
\end{prop}

\medskip
Next, we prove a simple fact which we are going to use three times when proving `absorbing lemmas'.
       \begin{prop} \label{l:absorbersampling}
          For all $s,t\in\bN$, $m \leq n^s$, and  $\alpha \in (0, 1)$,  let $\A_1, \dots, \A_m \subset [n]^t$ be families of $t$-tuples of the elements of $[n]$ with sizes $|\A_i|\ge4 \alpha t^2 n^{t}$, $i\in [m]$. Then, for sufficiently large $n$, there exists $\F\subset\bigcup_{i=1}^m\A_i$ such that $|\F|\le\alpha n$, the $t$-tuples in $\cal F$ are disjoint, and  $|\F \cap \A_i| \geq \alpha ^2 t^2 n / 4$ for all $i \in [m]$.
        \end{prop}

    \begin{proof}
        We use the probabilistic method. Let $\cR \subset [n]^t$ be a random family to which each $t$-tuple is sampled independently with probability  $p = \frac{\alpha n^{-t+1}}{4}$. Using Markov's inequality, we infer that
        $$\Prob(|\cR| \ge {\alpha n})\le\frac14$$ and, likewise, the probability that the number of intersecting pairs of $t$-tuples in $\cal R$ exceeds $ \alpha^2 t^2n/4$ is no more than $\frac14$. To see the latter, note that the expected number of intersecting pairs of $t$-tuples in $\cR$ is at most
        $$t^2 n^{2t-1}p^2 =\frac1{16}\alpha^2 t^2 n.$$
        Now consider $i \in [m]$. The random variable $|\cR \cap \A_i|$ is a binomially distributed random variable with expectation
        $$|\A_i|p\ge4\alpha t^2 n^tp=\alpha^2t^2n.$$
        Therefore, using the well-known Chernoff bound (see, e.g., \cite{JLR}), followed by the union bound over all $i\in[m]$, we infer that a.a.s.
        $$|\cR \cap \A_i| \geq \frac12\alpha^2t^2n\text{ \quad for all }i \in [m].$$
        Thus, the  probability that $\cal R$ satisfies all three  above properties is at least $\frac 12 - o(1)>0$, so the event is nonempty. Let  $\cal F'$ be an instance of $\cR$ satisfying them. Further, let ${\cal F}$ be obtained from $\F'$ by removing one $t$-tuple from each intersecting pair and disregarding all $t$-tuples of $\F'$ which do not belong to $\bigcup_{i=1}^m\A_i$. Clearly, we still have $|\F| \leq {\alpha n}$ and, moreover, for all $i \in  [m]$,
        $$|\F \cap \A_i|\ge\frac12\alpha^2t^2n- \frac14\alpha^2 t^2n=\frac14 \alpha^2 t^2 n,$$
        as required.
    \end{proof}

    \medskip

        A \emph{$(k,\ell)$-overlapping path}, or shortly \emph{$(k,\ell)$-path}, on $s$ vertices  is defined as an $s$-vertex $k$-graph whose vertices can be ordered linearly such that each of its edges consists of $k$ consecutive vertices and every two consecutive edges share exactly $\ell$ vertices. The \emph{length} of $P$ is defined as $|V(P)|=s$.
         Note that many vertex orderings yield the same path. We will typically fix one of them.

         Let $P$ be a $(k,\ell)$-overlapping path in $\cH$ on the vertices $ v_1, \dots, v_{s}$ (in this order).
        For $1\le p <q \le s$, we say that $P$ \emph{connects}, or \emph{lies between}, the segments $v_1, \dots, v_p$ and $v_{q}, \dots, v_s$.
        When $p=k$ and $q=s-k+1$, these two segments span edges which we call \emph{the end-edges} of $P$.

        For $\ell<k/2$, we will also need the notion of  \emph{$\ell$-ends} of $P$, defined as the $\ell$-sets $\{v_1, \dots,v_\ell\}$
        and $\{v_{s-\ell+1},\dots,v_s\}$.

        Our final `meta-theorem' will be used six times (twice in the proof of each of our three main results). In all these applications either the set $Q$ or its complement will be very small, but linear in $n$.

        \begin{prop}[Connecting Meta-Statement]\label{l:metatheorem} Let $(\cH,\phi)$ be an $n$-vertex colored $k$-graph, $1\le\ell<k$, and $Q\subset V:=V(\cH)$. For integers $g$ and $m$, where  $m=m(n)$, consider two statements.
        \begin{enumerate}
        \item[I.] For all $Q'\subset Q$, $|Q'|\le mg$, and all pairs of disjoint, \pc $(k,\ell)$-paths $P_1,P_2$ in $(\cH-Q,\phi)$, each having at least $\lceil\ell/(k-\ell)\rceil$ edges,  there exist vertices $v_1,\dots,v_{g'}\in Q\setminus Q'$, $g'\le g$, such that $P_1v_1\cdots v_{g'}P_2$ is also a \pc $(k,\ell)$-path in  $(\cH,\phi)$.
         \item[II.]   For every collection of  $m$ disjoint, \pc $(k,\ell)$-paths $P_1,\dots,P_m$ in $(\cH-Q,\phi)$, each having at least $\lceil\ell/(k-\ell)\rceil$ edges, there exist a \pc $(k,\ell)$-cycle $C$, and  a \pc $(k,\ell)$-path $P$  in $(\cH,\phi)$ which contain all paths $P_1,\dots,P_m$ and have at most $mg$ vertices in $R$. Moreover $P$ connects $P_1$ with $P_m$ and $P \subset C$.

\end{enumerate}
Then, Statement I implies Statement II.
\end{prop}

\proof Fixing $b$, in order to show the existence of a path $P$, we perform induction on~$m$. Trivially, it is true for $m=2$ (as it then follows from Statement I).
Assume there is a \pc path $P'$ which contains all paths $P_1,\dots,P_{m-1}$ and connects $P_1$ with $P_{m-1}$.
Set $Q'=V(P')\cap Q$ and note that $|Q'|\le (m-2)g$. Thus, we are in position to apply Statement I to $Q',P_{m-1},P_m$ obtaining a desired path $P$ which completes the induction. In order to obtain a cycle $C$,  we apply Statement I  to the pair $(P_m,P_1)$ with $Q'=V(P)\cap Q$. \qed

\bigskip

We finish this preliminary section with a couple of definitions related to the regularity method used in the proofs of `covering lemmas'.
Let  $V_1, \dots, V_k \subset V$ be mutually disjoint non-empty vertex sets of a $k$-graph $\cH$ on~$V$. We define the \emph{density} $d_{\cH} (V_1, \dots, V_k)$ of $\cH$ with respect to $(V_1, \dots, V_k)$ as the ratio of the number of edges in $\cH$ with one vertex in each $V_i$ to $|V_1| |V_2| \dots |V_k|$. We call the $k$-tuple $(V_1, V_2, \dots, V_k)$ $(\eps, d)$\emph{-regular} (for the hypergraph $\cal H$) if whenever $(A_1, \dots, A_k)$ is a $k$-tuple of subsets $A_i \subset V_i$ satisfying $|A_i| \geq \eps |V_i|$ for $i \in [k]$, we have
        $$|d_{\cH} (A_1, \dots, A_k) - d| \leq \eps.$$
        For $t \in \bN$, we call sets $V_1, \dots, V_t$ \emph{equitable} if $||V_i|-|V_j||\leq 1$ for all $i, j \in [t].$ A $k$-partite $k$-graph with equitable parts is also called \emph{equitable}.

\section{Tight Hamilton cycles with co-degree conditions} \label{sec:tight}
    This section is devoted to the proof of Theorem~\ref{tight}, so we will be constructing tight paths and cycles and the attribute `tight' will sometimes be omitted. We adapt the proof in~\cite{rrs08} to the  context of colored hypergraphs.

    Throughout the section,  given an integer $k\ge3$ and a sufficiently small $\gamma \in (0, 1)$, $(\cH,\phi)$ is an $n$-vertex colored $k$-graph on vertex set $V$ with \begin{equation}\label{setup}
    \delta_{k-1}(\cH)\ge(1/2+\gamma)n\quad\mbox{and}\quad\Delta_{k-1}(\cH_i)\le cn\quad\mbox{for all}\quad i\in\bN,
     \end{equation}
     where  $c>0$ is sufficiently small with respect to $\gamma$, while $n$ is sufficiently large. (In fact, in the actual proof of Theorem \ref{tight}, Lemmas \ref{l:reservoir} and \ref{l:pathcover} will be applied to large induced sub-hypergraphs $\cH'$ and $\cH''$ of $\cH$.

    We define the \emph{end-paths} of a tight path $P = v_1, \ldots, v_s$ of length $s\ge2k-2$ as the sequences $(v_1, \dots, v_{2k-2})$ and $(v_{s - 2k + 3}, \dots, v_{s})$, each spanning $k-1$ edges. This non-standard definition reflects  our frequent need to preserve the property of being properly colored for the union of two properly colored paths that share one of their end-paths.

We first state three lemmas from which the proof of Theorem \ref{tight} follows in a standard way. The first of them establishes the existence of an absorbing path.

    \begin{lemma}[Absorbing Lemma]\label{l:abspath}
        Let  $c=c(\gamma)>0$ be  
        sufficiently small.
        Then $(\cH,\phi)$ contains a \pc tight path $A$ with $|V(A)|\le\tfrac14(\gamma/2)^{k-2} n$ and such that for any  $U\subset V$, $|U| \leq \tfrac14(\gamma/2)^{2k} n$, there is a \pc tight path $A_U$ on the vertex set $V(A) \cup U$ having the same end-paths as $A$. (We will say that $A$ can absorb $U$.)
    \end{lemma}

    The next lemma sets aside a small set $R\subset V$, called \emph{reservoir}, which retains the properties of $(\cH,\phi)$ and is, hence, useful in connecting \pc paths into one almost hamiltonian tight cycle, without `interfering' with the previously built part. We remark that we could have requested the maximum degree in each color within $R$ to `scale' appropriately, but this is not necessary, since our proof requires $c \ll \gamma$ anyway.
    The role of the set $Z$ is to prevent $R$ from overlapping with $V(A)$.  The vertices of $P$ which do not belong to its ends are called \emph{inner}.

    \begin{lemma}[Reservoir Lemma]\label{l:reservoir}
        Let $\rho=\rho(\gamma)>0$ be sufficiently small and $0<c\le\frac{\gamma \rho}{6(k-1)}$. Then there is a set of vertices $R\subset V$ of size $ |R|\le\rho n$  such that for any $R'\subset R$ with $|R'|\leq \frac{\gamma }{10}|R|$ and any two disjoint  $(2k-2)$-tuples $\seq{v}, \seq{w} \in (V\setminus R)^{2k-2}$, both inducing \pc tight paths in $(\cH, \phi)$, there is in $(\cH, \phi)$ a \pc tight path of length at most $64k \gamma^{-2}$ with end-paths $\seq{v}$ and $\seq{w}$, whose all inner vertices belong to $R\setminus R'$.
    \end{lemma}

The third lemma allows us, given $A$ and $R$, to cover almost all vertices of the set $V\setminus(V(A)\cup R)$ by  a bounded number of disjoint, \pc paths.

    \begin{lemma}[Covering Lemma]\label{l:pathcover}
         For any $\delta>0$, let $ c=c(\gamma,\delta)>0$ be sufficiently small and $q=q(\gamma,\delta) $ be sufficiently large. Then there is a family $\cP$ of at most $q$ vertex-disjoint \pc tight paths in $(\cH,\phi)$ covering all but at most $\delta n$ vertices of~$\cH$.
    \end{lemma}

    The proof of Theorem~\ref{tight} follows the classical outline of the absorbing method (see \cite{rrs08} or \cite{rr-survey}).
		
\begin{proof}[Proof of Theorem~\ref{tight}]
Set
\begin{equation}\label{lgd} \lambda=\tfrac14(\gamma/2)^{2k}\;,\quad
\rho=\min\{\rho_{\ref{l:reservoir}}(7\gamma/8),\lambda/2\}\;,\quad\mbox{and}\quad
\delta=\min\{\delta_{\ref{l:pathcover}}(\gamma/2),\lambda/2\},
\end{equation}
where $\rho_{\ref{l:reservoir}}(7\gamma/8)$ and $\delta_{\ref{l:pathcover}}(\gamma/2)$ are given, respectively, by Lemma \ref{l:reservoir} and Lemma \ref{l:pathcover}, with, respectively, $7\gamma/8$ and  $\gamma/2$ instead of $\gamma$. Let $c_i$, $i=2,3,4$, be a constant $c$ yielded by, Lemma \ref{l:abspath}, Lemma \ref{l:reservoir}, and Lemma \ref{l:pathcover} respectively, with appropriately altered $\gamma$ -- see above.
Finally, let  $c=\min\{c_{\ref{l:abspath}},c_{\ref{l:reservoir}}/2,c_{\ref{l:pathcover}}/2\}$ and $q=q(\gamma/2,\delta)$ be as in Lemma~\ref{l:pathcover}, and let $n$ be sufficiently large.

Let  $(\cal H,\phi)$ satisfy (\ref{setup}) and let $A$ be a path provided by Lemma~\ref{l:abspath}, that is, a path of length
$$|V(A)|\le\tfrac14(\gamma/2)^{k-2} n\le \tfrac18\gamma n$$ which can absorb any set $U$ of up to $\lambda n$ vertices.

Apply Lemma~\ref{l:reservoir}  to $\cH'=\cH-V(A)$ with $\gamma:=7\gamma/8$ and $\rho$ and $c$ selected above. It is feasible as, setting $n'=|V(\cH')|=n-|V(A)|\ge(1-\gamma/8)n$,
$$\delta_{k-1}(\cH')\ge\delta_{k-1}(\cH)-|V(A)|)\ge\left(\frac12 + \frac78 \gamma  \right)n \ge\left(\frac 12 + \frac78 \gamma\right) n'$$
and
$$\Delta_{k-1}(\cH'_i)\le \Delta_{k-1}(\cH_i)\le cn\le\frac{c_{\ref{l:reservoir}}}{2(1-\gamma/8)}n'\le c_{\ref{l:reservoir}} n'.$$
 Let~$R$ be the resulting set of vertices of size $|R| \le\rho n'\le \lambda n'/2$.

Set $\cH''=\cH'-R$ and note that, since $\lambda\le \tfrac34\gamma$ and $n'':=|V(\cH'')|\ge(1-\gamma/2)n$,
    $$\delta_{k-1}(\cH'') \geq \left( \frac 12 + \frac78\gamma\right)n'-|R|\ge \left( \frac 12 + \frac78\gamma-\frac12\lambda\right)n'\ge\left( \frac 12 + \frac \gamma 2 \right)n''$$
    and
    $$\Delta_{k-1}(\cH''_i)\le \Delta_{k-1}(\cH_i)\le cn\le\frac{c_{\ref{l:pathcover}}}{2(1-\gamma/2)}n''\le c_{\ref{l:pathcover}} n''.$$
Apply Lemma~\ref{l:pathcover} to $(\cH'', \phi)$ with  $\delta\le\delta(\gamma/2)$ defined as in (\ref{lgd}), to obtain a family $\cal P$ of at most $q$ \pc paths which cover all but at most $ \delta n''\le \lambda n/2$ vertices of $\cH''$. Let $W$ be the set of all vertices of $\cH''$ not covered by the paths in $\cal P$. Set ${\cal P}^A= {\cal P}\cup\{A\}$.

To connect the paths from $\cal P^A$ into one cycle $C$, we  apply Proposition \ref{l:metatheorem} with $Q=R$ and $m=|{\cP^A}|=q+1$. To verify Statement I therein with $g=64k\gamma^{-2}-(4k-4)$, let $Q'$, $P_1$ and $P_2$ in $(\cH - R, \phi)$ be as in the statement. We invoke Lemma~\ref{l:reservoir} with $R'=Q'$ since, for $n$ large enough,
$$|Q'|\le mg\le (q+1)\cdot 64 k \gamma^{-2} \le\frac{\gamma}{10} |R|.$$
Indeed, the left-hand side of the second inequality above is a constant, while the right-hand side grows linearly with $n$. By Lemma \ref{l:reservoir}, applied to one end-path $\seq v$ of $P_1$ and one end-path $\seq w$ of $P_2$, we see that Statement I of  Proposition \ref{l:metatheorem} holds and thus so does Statement II.

        Let $U$ be the set of vertices of $\cH$ not covered by $C$. Since $U \subset R \cup W$, we have
        $$|U| \leq |R|+|W|\le \rho n+\delta n\le  \lambda n.$$
         Therefore, a tight Hamilton cycle can be constructed  in $\cH$ by absorbing $U$ into the absorbing path $A$, which is a sub-path of the cycle $C$. Since the new path $A_U$ which is replacing $A$ in~$C$ has the same end-paths as $A$, the obtained Hamilton cycle is \pc as well.
    \end{proof}


\subsection{Connecting and Reservoir Lemmas}
        The following connecting lemma is a key tool in constructing both the absorbing path and the reservoir.
        \begin{lemma}[Connecting Lemma]\label{l:con}
            Let  $(\cal H,\phi)$ satisfy (\ref{setup}). For any $c \le\frac{\gamma}{3(k-1)}$, a subset $V'\subset V$ with $|V'| \le\gamma n/10$,
              and any two disjoint $(2k-2)$-tuples $\seq v, \seq w \in (V\setminus V')^{2k-2}$, both inducing \pc tight paths in $(\cH, \phi)$, there is a \pc tight path $P$  in $(\cH, \phi)$ of length at most $8(2k-1)\gamma^{-2}$
             with end-paths $\seq{v}$ and $\seq{w}$, whose all  vertices belong to $V\setminus V'$.
        \end{lemma}


       Lemma \ref{l:con} will be deduced from an adaptation of Lemma 2.4 in \cite{rrs08}.  To state this adaptation, consider a \emph{directed} $k$-graph $\overrightarrow{\cH}$ where each edge is a \emph{$k$-tuple} (a sequence of vertices), rather than a set. For a sequence of vertices $\seq v$ and a vertex $u$ disjoint from $\seq v$, the two concatenations of $u$ and $\seq v$ will be denoted by $u \seq v$ and $\seq v u$. Given a directed $k$-graph $\overrightarrow{\cH}$ and a $(k-1)$-tuple $\seq v$ of its vertices, set
\[
d^+(\seq v) = | \{ u \in V: \seq v u \in \overrightarrow{\cH} \} | \quad \text{and} \quad d^-(\seq v) = | \{ u \in V: u \seq v  \in \overrightarrow{\cH} \} |.
\]
Furthermore,   define $d^{\pm}(\seq v) = \min(d^+ (\seq v), d^- (\seq v)) $ and call $\seq v$  \emph{extendable} if $d^{+}(\seq v)> 0$ \emph{or} $d^{-}(\seq v)> 0$.
A brief glance at the  proof of Lemma 2.4 in~\cite{rrs08} reveals that it goes through for any directed $k$-graph with high values of $d^{\pm}(\seq v)$ for every extendable $(k-1)$-tuple.
Consequently, the following version of that lemma is also true.
	
		\begin{lemma}[\cite{rrs08}]\label{l:con1}
			Let $\gamma > 0$ and $\overrightarrow \cH$ be a directed $n$-vertex $k$-graph. Assume that for any extendable $(k-1)$-tuple $\seq v$ of vertices of $\overrightarrow \cH$, $d^{\pm}(\seq v) \geq \left(\frac 12 + \gamma \right)n$.
			Then, for every  pair of disjoint, extendable $(k-1)$-tuples $\seq v $ and $\seq w$, there is a tight path in $\overrightarrow{\cH}$ of length at most $2k/\gamma^2$ between them. \qed
        \end{lemma}

       \bigskip

\begin{proof}[Proof of Lemma~\ref{l:con}]

 Given $(\cH,\phi)$, define an auxiliary \emph{directed}, $(2k-1)$-uniform hypergraph $\overrightarrow{\cH}=\overrightarrow{\cH}(\phi,V')$ with vertex set $V\setminus V'$ and the edge set consisting of sequences of distinct vertices $(v_1,\dots,v_{2k-1})\in (V\setminus V')^{2k-1}$ such that the sets $$\{v_1,\dots,v_k\},\dots,\{v_k,\dots,v_{2k-1}\}$$ are edges of $\cH$ with distinct colors. Equivalently, the edges of $\overrightarrow{\cH}$ correspond to the \pc tight paths of length $2k-1$ in $(\cH,\phi)$ with a fixed direction.

 By this construction, one cannot exclude the possibility that a $(2k-2)$-tuple $\seq v$ is not a segment of an edge of $\overrightarrow{\cH}$, and thus  $\max(d^+ (\seq v), d^- (\seq v)) =0$ (for instance, if $\seq v$ does not span a path in $\cH$). However, it turns out that  $\overrightarrow{\cH}$ satisfies the assumptions of Lemma~\ref{l:con1} with $k:=2k-1$ and $\gamma:=\gamma/2$.


Indeed, let $\seq v$ be a $(2k-2)$-tuple in $\overrightarrow{\cH}$ with $\max(d^+ (\seq v), d^- (\seq v)) >0$.
This means, in particular, that $\seq v$ spans in $(\cH,\phi)$ a \pc path.
Let $U$ be the set of vertices $u\in V\setminus V'$ such that $\{u, v_{k}, \dots, v_{2k-2} \} \in \cH$ and $\phi(u, v_{k}, \dots, v_{2k-2})$ does not appear on the consecutive $k$-tuples of $\seq v$. Clearly, $\seq v u$ is an edge of $\overrightarrow \cH$ for any $u \in U$, so $d^-(\seq v)= |U|$. Using the assumed bounds  $\delta_{k-1}(\cH)\ge(1/2+\gamma)n$ and $\Delta_{k-1}(\cH_i)\le c n$, we infer that
\[
d^{-}(\seq v) =|U| \geq \left(\frac12+\gamma-\frac{\gamma}{10}-(k-1)c\right)n  \ge \left(\frac 12 + \frac{\gamma}{2} \right)n,
\]
since $c\le\tfrac{\gamma}{3(k-1)}$ and $|V'|\le\gamma n/10$. Analogously, $d^+(\seq v) \ge \left(\frac 12 + \frac{\gamma'}{2} \right)n$.

To deduce Lemma~\ref{l:con}, consider vertex-disjoint $(2k-2)$-tuples $\seq v, \seq w \in (V\setminus V')^{2k-2}$ spanning \pc paths in $(\cH,\phi)$. Applying Lemma~\ref{l:con1} to $\overrightarrow{\cH}$ with $k:=2k-1$ and $\gamma = \gamma/2$, we get a path $P$ in $\overrightarrow{\cH}$ of length at most $8(2k-1)(\gamma) ^{-2}$ between $\seq v$ and $\seq w$.
It remains to check that $P$ corresponds to a properly colored tight path in $\cH$. Indeed, any two intersecting $k$-tuples contained in $P$ are also contained in some edge of $\overrightarrow{\cH}$, so they carry distinct colors by definition of  $\overrightarrow{\cH}$.

\end{proof}

\bigskip

\begin{proof}[Proof of Lemma~\ref{l:reservoir}]

Let $R$ be a set of vertices guaranteed by Proposition \ref{res} with $p=\tfrac23\rho$,  $U_S=N_{\cH}(S)$ for $S\in\binom V{k-1}$ and $\alpha_S=\tfrac12+\gamma$. In particular, for  large $n$, $\tfrac12\rho n\le |R|\le\rho n$ and, for each $S\in\binom V{k-1}$,
$$|N_{\cH}(S)\cap R|\ge \left(\tfrac12+\tfrac23\gamma\right)|R|.$$

We claim that $R$ fulfils the conclusion of Lemma~\ref{l:reservoir}. To see this, consider a set $R'\subset R$, $|R'|\le \gamma|R|/10$, and two disjoint $(2k-2)$-tuples of vertices $\seq v,\seq w\in (V\setminus R)^{2k-2}$ inducing \pc tight paths in $(\cH, \phi)$. Let
        $$\cR = \cH[(R  \cup\{\seq v, \seq w\}],$$
         where $\seq v$, $\seq w$ are viewed as sets. Set $r = |V(\cR)|$ and note that $\tfrac12\rho n\le|R|\leq r\le |R| + 4k $. Thus,
        \[
        \delta_{k-1}(\cR)\ge \delta_{k-1}(\cH[R]) \geq \left(\frac12 + \frac23\gamma\right)|R|  \geq \left( \frac 12 + \frac{\gamma}{2} \right) r
    \quad \text{and} \quad
        \Delta_{k-1}(\cR_i) \leq cn \leq \frac{2c}{\rho} r.
        \]
        Using Lemma~\ref{l:con} with $\cH := \cR$, $\phi := \phi|_{\cR}$, $\gamma:= \frac{\gamma}{2}$,  $c := 2c/\rho \leq \gamma/(3k)$, and $V'=R'$, we get the desired path of length at most $64 k \gamma^{-2}$ between $\seq v$ and $\seq w$.
    \end{proof}


\subsection{Absorbing path}

 For a vertex $v\in V$, a \emph{$v$-absorber} is a $(4k-4)$-tuple of vertices $\seq{w} = (w_1,\dots,w_{4k-4})$ spanning  a tight path $T$ in $\cH$  such that the sequence $ (w_1,\dots, w_{2k-2},v,w_{2k-1},\dots,w_{4k-4})$ spans another tight path $T_v$ in $\cH$. If both $T$ and $T_v$ are \pc paths,  we call the $(4k-4)$-tuple $\seq{w}$ a \emph{\pc $v$-absorber}.
 To construct an absorbing path, we first need to show that for every $v$ there are many, that is, $\Theta(n^{4k-4})$, \pc $v$-absorbers.


\begin{lemma}\label{l:many_tight_absorbers}
	Let $c \le 4(\gamma/2)^k(3k)^{-2}$. For every $v\in V(\cH)$ there are at least $4(\gamma/2)^kn^{4k-4}$ \pc $v$-absorbers in $(\cH,\phi)$.
\end{lemma}

\begin{proof}
	We begin by counting $v$-absorbers. A $v$-absorber $\seq{w}=(w_1, w_2,\ldots,w_{4k-4})$ can be constructed by sequentially selecting the vertices $w_i$, $i=1,\dots, 4k-4$, and each time counting the number of ways to do it. Initially, we will be allowing repetitions, $w_i=w_j$, as well as choices $w_i=v$. Those will be discarded at the end.
 \begin{itemize}
 \item
For $i=1,\dots,k-1$, there are no constraints on $w_i$, so the number of choices of $w_i$ is $n$.
 \item For $i=k,\ldots,2k-3$,  the $k$-tuple $\{w_{i-k+1},\dots,w_i\}$ must form an edge in $\cH$, so the number of choices of $w_i$ is
  at least $\delta_{k-1}(\cH)\geq (1/2+\gamma)n$.
 \item For $i=2k-2,\ldots,3k-3$,  vertex $w_i$ must belong to two edges, $\{w_{i-k+1},\dots,w_{i}\}$ and $\{v,w_{i-k+2},\dots,w_i\}$, so the number of choices of $w_i$ is at least
     $$2(1/2+\gamma)n-n=2\gamma n.$$
 \item For  $i=3k-2,\dots, 4k-4$, the $k$-tuple $\{w_{i-k+1},\dots,w_{i-1}\}$ must form an edge,  so the number of choices of $w_i$ is, again, at least $(1/2+\gamma)n>n/2$.
\end{itemize}
Altogether, as the number of choices with vertex repetitions is $O(n^{4k-5})$, for sufficiently large $n$, the total  number of $v$-absorbers in $\cH$ is at least
\[
n^{k-1}\left((1/2+\gamma)n\right)^{2k-3}(2\gamma n)^k-O(n^{4k-5}) \ge 8\left(\frac{\gamma}{2}\right)^kn^{4k-4}.
\]

Now we have to subtract the number of $v$-absorbers with a color conflict. Let $T$ and $T_v$ be the paths as in the definition of a $v$-absorber $\seq{w}$. Our task is to count the $v$-absorbers $\seq{w}$ in which either $T$ or $T_v$ contains two intersecting edges with the same color.

 Let us estimate from above the number of $v$-absorbers in which we have  $\phi(e)=\phi(f)$ for a given pair $(e,f)$ of edges in $T$. Let $w_j\in f\setminus e$. There are no more than $n^{4k-5}$ choices of the vertices $w_i$, $i=1,\dots,4k-4$, $i\neq j$. However, since $\Delta_{k-1}(\cH_{\phi(e)}) \leq cn$, and we want $\phi(e) = \phi(f)$, vertex $w_j$ can be chosen in at most $cn$ ways. Altogether, this gives us at most $cn^{4k-4}$ $v$-absorbers. By the same token, the number of $v$-absorbers in which we have  $\phi(e)=\phi(f)$ for a given pair $(e,f)$ of edges in $T_v$ is also  at most $cn^{4k-4}$.



By the union bound over all possible pairwise edge intersections in $T$ and in $T_v$ and using the assumption on $c$, the total number of \pc $v$-absorbers is at least
\[
8\left(\frac{\gamma}2 \right)^kn^{4k-4}-2 \cdot\binom{3k-2}2cn^{4k-4}\ge4\left(\frac{\gamma}2\right)^kn^{4k-4}.
\]
\end{proof}

\begin{proof}[Proof of Lemma~\ref{l:abspath}]
    For each $v\in V$, let $\cA_v$ be the family of all \pc $v$-absorbers in $(\cH,\phi)$.
    Based on  Lemma \ref{l:many_tight_absorbers}, we apply Proposition~\ref{l:absorbersampling} to the family $\{\cA_v: v\in V\}$ with $m=n$, $t=4k-4$, and
    $$\alpha =\frac{(\gamma/2)^k}{(4k-4)^2}.$$
As an outcome, we obtain a family $\cal F$ of vertex disjoint \pc paths of length $(4k-4)$ with $|\F| \le \alpha n$ and such that for each vertex $v$ there are at least
    \begin{equation}\label{vabs}
    \alpha^2(4k-4)^2 n/4=\frac14\left(\frac{\gamma}2\right)^{2k}n
    \end{equation} \pc $v$-absorbers in $\F$.

    To connect the paths from $\F$ into one path $A$, we  apply Proposition \ref{l:metatheorem} with $Q=V\setminus\bigcup_{F\in\F}V(F)$ and $m=|\F|$. To verify Statement I therein, we  invoke Lemma \ref{l:con} with $Z=Q'\cup \bigcup_{F\in\F}V(F)$, so we set  $g=8(2k-11)\gamma^{-2}-(4k-4)$. Note that
    \[
    |Z|\le |Q'|+|\F|(4k-4)\le \alpha n \cdot \frac{8(2k-1)}{\gamma^2} \le \frac{16k}{\gamma^2}\alpha n = \frac{\gamma^{k-2}k}{2^k(k-1)^2}n<\frac\gamma{10}n.
    \]
    By Lemma \ref{l:con}, applied to one end-path $\seq v$ of $P_1$ and one end-path of $\seq w$ $P_2$, we see that Statement I of  Proposition \ref{l:metatheorem} holds  and thus Statement II follows. The obtained path $A$ has length
    $$|V(A)|\le\alpha n \cdot \frac{8(2k-1)}{\gamma^2}\le\frac{\gamma^{k-2}k}{2^k(k-1)^2}n<\frac14(\gamma/2)^{k-2} n,$$
    as required.

    Recall that, by (\ref{vabs}), for each vertex $v\in V$, the path $A$ contains at least $\frac14\left(\frac{\gamma}2\right)^{2k}n$
    vertex-disjoint \pc $v$-absorbers. Therefore, one can absorb into $A$ any set of vertices $U$ of size $|U|\le \frac14\left(\frac{\gamma}2\right)^{2k} n$, one by one, obtaining a new \pc path $A_U$. After absorbing a vertex $v$ into $A$ via a $v$-absorber $\seq{w}$, only the edges containing $w_{k}, \dots, w_{3k-3}$ are reconfigured, so the new path is properly colored, has the same end-paths, and all the other absorbers remain unaffected. In particular, the final path $A_U$ has the same end-paths as~$A$.
\end{proof}


    \subsection{Covering by long paths}


Our proof of Lemma~\ref{l:pathcover} follows that in \cite{rrs08}, Section 4. The proof in \cite{rrs08} relied on five technical claims and only the first two of them require, due to the coloring, certain modifications. Claim~\ref{l:densepath} below is an analogue of  Claim 4.1 in \cite{rrs08}.

\begin{claim} \label{l:densepath}
    Let $(\cH, \phi)$ be a colored $k$-partite $k$-graph with at most $m$ vertices in each part and at least $dm^k$ edges. If $\Delta_{k-1}(\cH_i) \leq \frac{d}{2k^2}m$ for all $i\in\bN$, then $\cH$ contains a \pc tight path on at least $\frac d2 m$ vertices.
\end{claim}

\begin{proof}
    Denote the parts of $\cH$ by $V_1, \dots, V_k$, and call a set $S \subset V(\cH)$ \emph{relevant} if $|S| = k-1$ and $|S \cap V_i| \leq 1$ for $i \in [k]$. Note that there are at most $km^{k-1}$ relevant sets. We start by preprocessing $\cH$ as follows. If there is a relevant set $S$ whose degree in the \emph{current} hypergraph is smaller than $\frac{dm}{k}$, delete all edges containing $S$. Repeat this step until all relevant sets have degree either $0$ or at least $\frac{dm}{k}$. Denote the resulting hypergraph by $\cH'$. Observe that  the number of deleted edges is strictly smaller than $\frac{dm}{k} \cdot km^{k-1} = dm^k$ since the edges containing each relevant set were removed at most once. Hence $\cH'$ is non-empty.

    Let $P$ be the longest \pc tight path in $\cH'$. Denote its vertices by $v_1, \dots, v_\ell$ in that order, and observe that since $P$ is tight, any part $V_i$ contains precisely every $k$-th vertex of $P$. Let $S= \{v_1, \dots, v_{k-1}\}$, and let $C$ be the set of colors appearing on the $k-1$ edges of $P$ intersecting $S$. Let $U$ be the set of vertices $v$ for which $e_v:=S\cup\{v\}$ is an edge of $\cH'$ with $\phi(e_v) \notin {C}$. By construction of $\cH'$ and the assumption on $\phi$,
        $$|U|\geq \frac{dm}{k}-\frac{(k-1)dm}{2k^2} \geq \frac{dm}{2k}.$$
    On the other hand, by maximality of $P$, $U \subset V(P)$. Notice that there is $j\in[k]$ such that $U\subset V_j$, and thus, each vertex $v \in U$ has $k-1$ predecessors on $P$ not belonging to $U$. Therefore $|V(P)| \geq k |U| \geq \frac{dm}{2},$ as required.
\end{proof}

\medskip

Our next result is a suitably modified version of Claim 4.2 in \cite{rrs08}.

\begin{claim} \label{l:manypath}
    For all $0<\eps<d<1$, every $\eps$-regular, equitable $k$-partite $k$-graph $\cH$ on $n$ vertices
    with density $d_{\cH}\ge d$ and with  $\Delta_{k-1}(\cH_i) \leq \frac{\eps d}{2k^3}n$ for all $i\in\bN$, contains a family ${\cal P}$ of vertex-disjoint \pc paths such that
    $$\mbox{for each } P\in{\cal P}\mbox{ we have }|V(P)|\ge\frac{\eps(d-\eps)n}{2k}\quad\mbox{and}\quad\sum_{P\in{\cal P}}|V(P)|\ge(1-2\eps) n.$$
\end{claim}

\proof
    Let $\cal P$ be a largest family of vertex-disjoint \pc  paths with $|V(P)|\ge\eps(d-\eps)n/(2k)$ for each $P\in{\cal P}$. Suppose that $\sum_{P\in{\cal P}}|V(P)| < (1-2\eps) n$.
    The proof goes along the lines of the one in \cite{rrs08} (with $\alpha:=d$ and $\cal Q:=\cal P$), except for the very end, where we focus on a sub-$k$-graph $\hat \cH$ with $m\ge \eps n/k$ vertices in each partition class which is vertex disjoint from all paths in $\cal P$.
    Here we have to note that
    $$\Delta_{k-1}(\hat \cH_i)\le\Delta_{k-1}(\cH_i) \leq \frac{\eps d}{2k^3}n=\frac{ d}{2k^2}(\eps n/k)\le\frac{ d}{2k^2}m.$$
    Moreover, by the $\eps$-regularity of $\cH$, we have $|\hat \cH|\ge (d-\eps)m^k$, so Claim \ref{l:densepath} can be applied to $\hat \cH$, producing a \pc path in $\cH$ of length at least $\eps(d-\eps)n/(2k)$, vertex disjoint from $\cal P$. This yields a contradiction with the maximality of $\cal P$ (for details see \cite{rrs08}.) \qed

\medskip

The remaining three claims from the proof in \cite{rrs08}, combined together, imply the existence of a vertex-decomposition of our hypergraph into $\eps$-regular $k$-tuples (plus some leftover vertices). This was achieved by using the Weak Hypergraph Regularity Lemma. Notice that this statement concerns hypergraphs only, making no mention of the coloring. Given an $\eps$-regular partition $(V_1,\dots,V_t)$ of a $k$-graph $\cH$ and a real $d>0$, we denote by
$\K(V_1,\dots,V_t;d,\eps)$ the $k$-graph on vertex set $[t]$ where edges correspond to $\eps$-regular $k$-tuples $(V_{i_1},\dots, V_{i_k})$ with density at least $d$. Our last claim corresponds to Claims 4.3-4.5 in \cite{rrs08}.

\begin{claim} \label{l:regtriples_tight}
    Given $\gamma>0$ and sufficiently small $\eps=\eps(\gamma) >0$, there exist integers $T_0$ and $n_0$ such that the following holds. Any $n$-vertex $k$-graph $\cH$, $n\ge n_0$, with $\delta_{k-1}(\cH) \geq \left( \frac {1}{2} + \gamma \right) n$ contains an equitable partition $V=V_1\cup\dots\cup V_t$, $t\le T_0$, such that the corresponding $k$-graph $\K=\K(V_1,\dots,V_k;\tfrac14,\eps)$ possesses a matching covering at least $\left(1-2k\eta\right)t$ vertices, where $\eta=k(\sqrt\eps)^{1/(k-1)}$. \qed
\end{claim}

\begin{proof}[Proof of Lemma~\ref{l:pathcover}]
    For any $\gamma>0$ and $\delta>0$, let $\eps$ satisfy $2\eps+4k\eta\le\delta.$ Further, let $T_0=T_0(\eps)$ and $n_0=n_0(\eps)$ be as in Claim \ref{l:regtriples_tight} and set
    $$c=\frac{\eps}{16k^2T_0}\quad\mbox{and}\quad q=\frac{16T_0}{\eps}.$$

    Let $\cH$ be an $n$-vertex $k$-graph, $n\ge n_0$, with
    $$\delta_{k-1}(\cH) \geq \left( \frac {1}{2} + \gamma \right) n\quad\mbox{and}\quad\Delta_{k-1}(\cH_i) \leq cn.$$
    By Claim \ref{l:regtriples_tight}, there is an equitable partition $V=V_1\cup\dots\cup V_t$, $t\le T_0$, such that the corresponding $k$-graph $\K=\K(V_1,\dots,V_t;\tfrac14,\eps)$ possesses a matching $\cal M$ covering at least $\left(1-2k\eta\right)t$ vertices of $\K$. Note that
    $$n/2T_0\le \lfloor n/t\rfloor\le|V_i|\le \lceil n/t\rceil.$$
    For each $e\in \cal M$, let $\cH_e$ be the sub-$k$-graph of $\cH$ induced by the partition sets constituting the edge $e$ of $\K$. We have
    $$\Delta_{k-1}((\cH_e)_i) \leq\Delta_{k-1}(\cH_i) \leq\frac{\eps}{8k^2}\left(\frac n{2T_0}\right)\le\frac{\eps}{8k^3}\left(k\lfloor n/t\rfloor\right)\le \frac{\eps}{8k^3}|V(\cH_e)|.$$

    Thus, by Claim \ref{l:manypath} applied to $\cH:=\cH_e$ with $d=1/4$, there is a family ${\cal P}_e$ of vertex-disjoint \pc  paths such that
    $$\mbox{for each } P\in{\cal P}_e\mbox{ we have }|V(P)|\ge\frac{\eps(1/4-\eps)n}{2T_0}\quad\mbox{and}\quad\sum_{P\in{\cal P}_e}|V(P)|\ge(1-2\eps)|V(\cH_e)|.$$

    Applying the same argument to each $e \in {\cal M}$ gives a collection of paths $\mathcal{P}$ which cover all but
    at most $(2\eps+4k\eta)n\le\delta n$ vertices of $\cH$. Clearly,
    $$|{\cal P}| \leq \frac{2T_0}{\eps(1/4-\eps)}\le\frac{16T_0}{\eps}.$$
 \end{proof}


\section{Hamilton $(k,\ell)$-cycles with co-degree conditions}  \label{sec:ell-cycle}

This section is devoted to the proof of Theorem~\ref{ell-cycles}.  Recall that given a $k$-graph $\cH$ and a coloring $\phi$ of its edges,  $\cH_i=\{e\in\cH: \phi(e)=i\}$.
Throughout this section,
given integers $1\le\ell<k/2$,
 $(\cH,\phi)$ is an $n$-vertex colored $k$-graph on vertex set $V$ with $n$ sufficiently large and divisible by $k-\ell$, and such that
\begin{equation}\label{setup2}
    \Delta_{\ell}(\cH_i)\le cn^{k-\ell}\quad\mbox{for all}\quad i\in\bN,
     \end{equation}
     where  $c>0$ is sufficiently small. As all  our statements below are about $(k,\ell)$-paths and $(k,\ell)$-cycles, we will often call them just paths and cycles. As in previous section, we will first state three lemmas from which Theorem \ref{ell-cycles} follows.

     Observe that, since $2 \ell < k$, the degree of any vertex  on a path or a cycle is either  one or two. Recall that for  $\ell<k/2$, the $\ell$-ends of a $(k,\ell)$-path are defined as the sets of the first and the last $\ell$ vertices of the path.

	\begin{lemma}[Absorbing Lemma] \label{l:abspath_ell-cycle}
        There exist  constants $a=a(k,\ell)<1$ and $b= b(k,\ell)<1$
        such that for every $\lambda>0$ and $c\le  b \lambda^5$, if $(k-\ell)|n$, $\delta_{k-1}(\cH)\ge\lambda n$ and (\ref{setup2}) holds, then
        there is  a \pc  $(k,\ell)$-path $A$ in $\cH$ with $|V(A)|\le\lambda^5 n$ such that for every $U\subset V$ with $|U| \le a(k-\ell)\lambda^{10} n$ and $(k-\ell)||U|$,  there is a \pc $(k,\ell)$-path $P_U$ in $\cH$ with  $V(P_U)=V(A) \cup U$  and with the same end-edges and $\ell$-ends as $A$. (We will say that $A$ can absorb $U$.)
    \end{lemma}

\begin{lemma}[Reservoir Lemma] \label{l:reservoir_ell-cycle}
         For every  $0<\rho <1$, $d>0$, and $c \leq\tfrac{(d/2)^3}{216k!^3}(\rho/2)^{k-\ell}$, if $\delta_{k-1}(\cH)\ge dn$ and (\ref{setup2}) holds, then there is a set $R \subset V$ with $  |R| \leq \rho n$ such that
        for any two disjoint $\ell$-sets $X,Y\subset V\setminus R$, any two colors $c_X, c_Y$, and any subset $R'\subset R$, $|R'|\le d|R|/20$, there exists a 3-edge \pc $(k,\ell)$-path $P=X,v_1\dots v_{3k-4\ell},Y$ connecting $X$ and $Y$, with $\{v_1\dots v_{3k-4\ell}\}\subset R\setminus R'$ and such that
         $\phi(X,v_1\cdots v_{k-\ell})\neq c_X$ and $\phi(v_{2k-3\ell+1}\cdots v_{3k-4\ell},Y)\neq c_Y$.
    \end{lemma}

\begin{lemma}[Covering Lemma] \label{l:pathcover_ell-cycle}
For every $\gamma>0$ and $\delta>0$, let $ c=c(\gamma,\delta)>0$ be sufficiently small and $q=q(\gamma,\delta) $ be sufficiently large. If $\delta_{k-1}(\cH)\ge\left(\tfrac1{2(k-\ell)}+\gamma\right)n$ and (\ref{setup2}) holds, then there is a family $\cP$ of at most $q$ vertex-disjoint \pc  $(k,\ell)$-paths in $(\cH,\phi)$ covering all but at most $\delta n$ vertices of~$\cH$.
    \end{lemma}

\begin{proof}[Proof of Theorem~\ref{ell-cycles}] Let $\gamma>0$ be given and let $a$ and $b$ be the constants provided by Lemma \ref{l:abspath_ell-cycle}.
    Set
    $$\lambda = (\gamma/4)^{1/5},\quad\rho=\delta=\frac12a\lambda^{10}=\frac{a\gamma^2}{32},\quad d\le\frac1{2(k-\ell)},\quad q=q_{\,\ref{l:pathcover_ell-cycle}}(\gamma/2,\delta),$$
     and
     $$c = \min\left(\nak{b}\lambda^5, \ \frac{1}{217k!^3}\cdot \left(\frac{d }{2 } \right)^{3}\!\left(\frac{ \rho}{2 } \right)^{k-\ell}, \  \frac12c_{\,\ref{l:pathcover_ell-cycle}}(\gamma/2,\delta)\right),$$
      where the subscript $_{\ref{l:pathcover_ell-cycle}}$ means that the constant comes from Lemma \ref{l:pathcover_ell-cycle}. Note that the second and third  ingredients of the minimum above incorporate an extra margin to accommodate the forthcoming estimates.

Let $(\cH,\phi)$ be an $n$-vertex colored $k$-graph on vertex set $V$ with $n$ sufficiently large and divisible by $k-\ell$. Assume that $\delta_{k-1}(\cH)\ge\left(\tfrac1{2(k-\ell)}+\gamma\right)n$. Since, for sufficiently small $\gamma=\gamma(k,\ell)$ we have $\tfrac1{2(k-\ell)}+\gamma\ge\tfrac1{2(k-\ell)}\ge \lambda$,
 we may apply Lemma~\ref{l:abspath_ell-cycle} to $(\cH,\phi)$  to find an absorbing path $A$ with $|V(A)|\le\lambda^5n=\gamma n/4$ which can absorb any set $U$ of up to $a(k-\ell)\lambda^{10}n$ vertices.

    Next, let $R\subset V \setminus V(A)$ be the set given by Lemma~\ref{l:reservoir_ell-cycle} applied to $\cH'=\cH-V(A)$ and to the coloring $\phi'$ induced in $\cH'$ by $\phi$. Such an application is feasible since, setting $n'=|V(\cH')|=n-| V(A)|\ge(1-\gamma/4)n$, we have
    $$\delta_{k-1}(\cH')\geq \delta_{k-1}(\cH)-|V(A)|n\ge\frac n{2(k-\ell)} \geq dn'$$
    and
    $$\Delta_\ell(\cH'_i)\le\Delta_\ell(\cH_i)\le cn^{k-\ell}\le\frac c{(1-\gamma/4)^{k-\ell}}(n')^{k-\ell}\leq \frac{1}{216k!^3}\cdot \left(\frac{d }{2 } \right)^{3}\!\left(\frac{ \rho}{2 } \right)^{k-\ell} (n')^{k-\ell}.$$

        Further, let $\cH''=\cH'-R$  and let $\phi''$ be the coloring induced in $\cH''$ by $\phi$.
        Set $n'':=|V(\cH'')|$. Since  $\gamma/4+a\gamma^2/32\le\gamma/2$,
        $$\delta_{k-1}(\cH'') \geq \delta_{k-1}(\cH)-(|V(A)|+|R|)n\ge\left(\frac1{2(k-\ell)}+\gamma/2\right)n''$$
        and
    $$\Delta_\ell(\cH''_i)\le\Delta_\ell(\cH_i)\le cn^{k-\ell}\le\frac{c_{\,\ref{l:pathcover_ell-cycle}}(\gamma/2,\delta)}{2(1-\gamma/2)^{k-\ell}}(n'')^{k-\ell}\leq c_{\,\ref{l:pathcover_ell-cycle}}(\gamma/2,\delta)(n'')^{k-\ell}.$$

        Hence, we can apply Lemma~\ref{l:pathcover_ell-cycle} to the colored hypergraph $(\cH'', \phi'')$ with $\gamma:=\gamma/2$ to obtain a family $\cal P$ of at most $q$ \pc paths which cover all but at most $ \tfrac12a\lambda^{10}n$ vertices of $\cH''$. Let $W$ be the set of vertices of $\cH''$ not covered by any of the paths in $\cP$. Include the absorbing path $A$ into $\cP$ to get a family of paths $\cP^A = \cP\cup\{A\}$.

To connect the paths from $\cP^A$ into one cycle $C$, we  apply Proposition \ref{l:metatheorem} with $Q=R$ and $m=|{\cP^A}|=q+1$. To verify Statement I therein with $g=3k-4\ell$, let $Q'$, $P_1$ and $P_2$ in  be as in the statement. We invoke Lemma \ref{l:reservoir_ell-cycle} with $R'=Q'$ since, for $n$ large enough,
$$|Q'|\le mg\le (q+1)(3k-4\ell)\le\frac d{20}|R|.$$
 By Lemma \nak{\ref{l:reservoir_ell-cycle}}, applied to one $\ell$-end $X$ of $P_1$ and one $\ell$-end $Y$ of  $P_2$, we see that Statement~I of  Proposition \ref{l:metatheorem} holds and thus so does Statement II.

        Let $U$ be the set of vertices of $V$ not covered by $C$. As $|V(C)|$ is divisible by $k-\ell$, so is $|U|$.
        Since
        $$|U| \le|W|+|R|\le\delta n+\rho n\le a\lambda^{10}n\le a(k-\ell)\lambda^{10}n,$$ $U$ can be absorbed into $A$ by replacing $A$ with a \pc path $P_U$. Since $P_U$ and $A$ have the same $\ell$-ends, the absorption extends $C$ to a Hamilton cycle in $\cH$. Since, in addition, $P_U$ and $A$ have the same end-edges, the obtained Hamilton cycle remains properly colored.
    \end{proof}

\subsection{Connecting and Reservoir Lemmas} In this subsection we  prove first a simple connecting lemma which is used in the proofs of both the Absorbing and the Reservoir Lemma.
It says, roughly, that  any two disjoint $\ell$-sets  can be connected  via a short, \pc  path which avoids a given relatively small set of vertices. In fact, although we do not need it here, we show that there are many such paths.

\begin{lemma}[Connecting Lemma]\label{l:connecting_ell-cycle} For every $\kappa>0$ and $c\le\tfrac1{216} \kappa^3 k!^{-3}$,   let a colored $k$-graph $(\cH,\phi)$ be given with
$\delta_{k-1}(\cH)\ge \kappa n$ and $\Delta_{\ell}(\cH_i)\le cn^{k-\ell}$. Further, let  $V'\subset V$, $|V'|\leq \kappa n/10$. Then, for sufficiently large $n$, for any pair of disjoint $\ell$-sets $X,Y\subset V\setminus V'$ and any two colors $c_X,c_Y$, there exist at least $\kappa^3 n^{3k-4\ell}/(54k!^3)$ 3-edge \pc paths  $Xv_1\cdots v_{3k-4\ell}Y$, where $v_1,\dots,v_{3k-4\ell}\not\in V'$, such that, in addition, $\phi(X,v_1,\dots,v_{k-\ell})\neq c_X$ and  $\phi(v_{2k-3\ell+1},\dots,v_{3k-4\ell},Y) \neq c_Y$.
\end{lemma}

\begin{proof} We will first estimate from below the number of paths $Xv_1\cdots v_{3k-4\ell}Y$ with no regard to coloring. Then we will subtract an upper bound on the number of those among them which have a color conflict. Throughout we assume that $n$ is large enough for all the estimates to hold.

Note that by the assumption  $\delta_{k-1}(\cH)\ge \kappa n$,  every set $Z\subset V$ of size $|Z|=s<k$ is contained in at least
\begin{equation}\label{alfa}
\binom{n-s}{k-1-s}\frac{\kappa n}{k-s}\ge\frac{\kappa n^{k-s}}{2k!}
 \end{equation}
 edges of $\cH$. Now, fix two disjoint $\ell$-sets $X,Y\subset V\setminus V'$. By \eqref{alfa}, there are at least $ \tfrac{\nak{\kappa} n^{k-\ell}}{3k!}$  edges $X'$ of $\cH$ such that $X\subset X'$ and $X'\cap (V'\cup Y)=\emptyset$. Similarly, there are at least $\tfrac{\nak{\kappa} n^{k-\ell}}{3k!}$  edges $Y'$ of $\cH$ such that $Y\subset Y'$ and $Y'\cap (V'\cup X')=\emptyset$. For a fixed pair $(X',Y')$, select arbitrarily subsets $Z_X\subset X'\setminus X$ and $Z_Y\subset Y'\setminus Y$ of size $|Z_X|=|Z_Y|=\ell$. Set $Z=Z_X\cup Z_Y$ and notice that $|Z|=2\ell<k$. Thus, there are at least $\tfrac{\nak{\kappa}n^{k-2\ell}}{3k!}$ edges $T$ of $\cH$ such that $Z\subset T$ and $T\cap(X'\cup Y')=\emptyset$.
Altogether, there are at least $\tfrac{\nak{\kappa^3}n^{3k-4\ell}}{27k!^3}$ choices of $(X',T,Y')$ and each of them corresponds to at least one 3-edge path between $X$ and $Y$.

Now we count the paths with color conflicts. In addition to  $\phi(X,v_1\dots,v_{k-\ell})= c_X$ and $\phi(v_{2k-3\ell+1},\dots v_{3k-4\ell},Y) = c_Y$,
there are two more potential conflicts, corresponding to the two pairs of edges intersecting, respectively, at $Z_X$ and at $Z_Y$. Using the assumption $\Delta_\ell(\cH_i)\le cn^{k-\ell}$  each one of them disqualifies at most $cn^{3k-4\ell}$ paths. We conclude that  the number of required paths is at least
$$\frac{\nak{\kappa^3} n^{3k-4\ell}}{27k!^3}-4cn^{3k-4\ell}\ge \frac{\nak{\kappa^3} n^{3k-4\ell}}{54k!^3},$$
by our assumption on $c$.
\end{proof}

\bigskip

\begin{proof}[Proof of Lemma~\ref{l:reservoir_ell-cycle}]

Let $R$ be a set of vertices guaranteed by Proposition \ref{res} with $p=\tfrac23\rho$,  $U_S=N_{\cH}(S)$ for $S\in\binom V{k-1}$ and $\alpha_S=d$. In particular, for  large $n$, $\tfrac12\rho n\le |R|\le\rho n$ and, for each $S\in\binom V{k-1}$,
$$|N_{\cH}(S)\cap R|\ge \frac23d|R|.$$

We claim that $R$ fulfils the conclusion of Lemma~\ref{l:reservoir_ell-cycle}. To see this, consider a set $R'\subset R$, $|R'|\le d|R|/20$, two disjoint $\ell$-tuples of vertices $X,Y\subset V\setminus V'$ and colors  $c_X,c_Y$. Let
        $\cR = \cH[(R  \cup X\cup Y],$  set $r = |V(\cR)|$ and note that $\tfrac12\rho n\le|R|\leq r\le |R| + 2\ell $. Thus, for large $n$,
        \[
        \delta_{k-1}(\cR) \geq \frac23d|R|  \geq \frac12dr
    \quad \text{and} \quad
        \Delta_{k-\ell}(\cR_i) \leq cn^{k-\ell} \leq c\left(\frac{2}{\rho}\right)^{k-\ell} r^{k-\ell}.
        \]
        Using Lemma~\ref{l:connecting_ell-cycle} with $\cH := \cR$, $\phi := \phi|_{\cR}$, $\kappa:= \frac d{2}$,
        $$c := c\left(\frac{2}{\rho}\right)^{k-\ell} \leq \left( \frac d2 \right)^3 \cdot \frac{1}{216k!^3}\left(\frac{\rho}{2}\right)^{k-\ell},$$ and $V'=R'$, we get the desired path of length three between $X$ and $Y$.
    \end{proof}

\subsection{Absorbing path}

In this subsection we prove  Lemma~\ref{l:abspath_ell-cycle}. As usual, the absorbing path will be built from small pieces called \emph{absorbers}.
For convenience, paths are here represented  by sequences of edges (with an implicit ordering of the vertices).
Given a set $S\in \binom{V}{k-\ell}$, an \emph{S-absorber} is a 3-edge  path $P = (E_1, G, E_2)$ along with its $\ell$-ends $F_1$ and $F_2$ for which there exists a 4-edge path $Q = (E_1, G_1, G_2, E_2)$ with $V(Q) = V(P) \cup S$ whose  $\ell$-ends are $F_1$ and $F_2$.

Proposition 9 in \cite{hs10} and the initial part of the proof of Lemma 5 in \cite{hs10} together imply the following lower bound on the number of $S$-absorbers. (When citing that result, we  use $\lambda$ in place of $\varepsilon$.)

\begin{prop}[\cite{hs10}]\label{abs_from_HS}
For all $\lambda>0$, if $\delta_{k-1}(\cH)\ge\lambda n$, then for every $S\in\binom V{k-\ell}$ there are at least
\begin{equation}\label{zeta}
\frac{\lambda^5(3k-4\ell)!}{2^{6+3k}k^4(3k-2\ell)!}\binom n{3k-2\ell}=:\zeta\binom n{3k-2\ell}\ge \frac{\zeta}{2(3k-2\ell)!}n^{3k-2\ell}
 \end{equation}
 $S$-absorbers in $\cH$. \qed
\end{prop}

An $S$-absorber is called \emph{\pc} if both paths, \nak{$E_1, G,E_2$} and  \nak{$E_1,G_1,G_2,E_2$}, are properly colored. Note that there are two intersecting pairs of edges in $E_1,G,E_2$ and three in $E_1,G_1,G_2,E_2$.
Thus, arguing as in the proof of Lemma \ref{l:connecting_ell-cycle}, there are no more than $5cn^{3k-2\ell}$ $S$-absorbers which are not properly colored.

\begin{cor}\label{c}
For $\zeta$ as in (\ref{zeta}) and $c\le\tfrac\zeta{20(3k-2\ell)\nak{!}}$, for every $S\in\binom V{k-\ell}$ there are at least $\tfrac{\zeta}{4(3k-2\ell)!} n^{3k-2\ell}$ \pc $S$-absorbers. \qed
\end{cor}

\medskip

\begin{proof}[Proof of Lemma~\ref{l:abspath_ell-cycle}] We are going to prove Lemma~\ref{l:abspath_ell-cycle} with
$$a=a(k,\ell)=\frac{(3k-4\ell)!^2}{2^{22+6k}k^8(3k-2\ell)!^4(3k-2\ell)^2}\quad\mbox{and}\quad b=b(k,\ell)=\frac{(3k-4\ell)!}{5\cdot2^{8+3k}k^4(3k-2\ell)!^2}.$$
Note that, with this $b$, the bound on $c$ in Corollary \ref{c} coincides with that in Lemma~\ref{l:abspath_ell-cycle}.

 We apply Proposition~\ref{l:absorbersampling}, with parameters $t=3k-2\ell$ and $\alpha=\zeta(2t!t^2)^{-1}$, where $\zeta$ is as in~(\ref{zeta}), to the families $\cA_S$ of \pc $S$-absorbers viewed as vertex sequences. Note that, by Corollary \ref{c}, we have, indeed, $|A_S| \geq 4\alpha t^2 n^t$.  As an outcome, we obtain a family $\F$ of disjoint absorbers, i.e. members of $\bigcup_S\cA_S$, of size $\F\le \alpha n$ and such that for all $S$ we have
$$|\F\cap\A_S|\ge\alpha^2t^2 n/4=\frac{\zeta^2}{16t!t^2}n=a\lambda^{10}n.$$

 To connect the paths from $\F$ into one path $A$, we  apply Proposition \ref{l:metatheorem} with $Q=V\setminus\bigcup_{F\in\F}V(F)$ and $m=|\F|\le\alpha n$. To verify Statement I therein with $g = 3k - 4\ell$, let $Q' \subset Q$ and $P_1, P_2$ in $(\cH - Q, \phi)$ be as in the statement. We  invoke Lemma \ref{l:connecting_ell-cycle} with $\kappa=\lambda$ and $V'=Q'\cup \bigcup_{F\in\F}V(F)$ satisfying
    \[
    |V'|\le |Q'|+|\F|(3k-2\ell)\le mg+ \alpha n(3k-2\ell)\le6(k-\ell)\alpha n\le\zeta\le\lambda^5\le\lambda n/10.
    \]
     \nak{Moreover, the upper bound  $c \leq b \lambda^5$ is stronger than that required for Lemma~\ref{l:connecting_ell-cycle}}.
       Applying the lemma to one $\ell$-end $X$ of $P_1$ and one $\ell$-end $Y$ of $P_2$, we see that Statement~I of  Proposition \ref{l:metatheorem} holds and thus Statement II follows. Note that the obtained path~$A$ has length
    $$|V(A)|\le |\F|(6k)\le \zeta n\le \lambda^5 n,$$
    as required.

 Finally, for every $U\subset V$ of size $|U|\le a(k-\ell)\lambda^{10}n$ where $(k-\ell)||U$, we may partition it into sets of size $k-\ell$, say
$U=S_1\cup\cdots\cup S_u$, where $u=|U|/(k-\ell)\le a\lambda^{10}n$. Since for each $i=1,\dots,u$, there are at least $a\lambda^{10}n$ disjoint \pc $S_i$-absorbers on~$A$, we can greedily absorb each set $S_i$ onto $A$, obtaining a new $(k,\ell)$-path $P_U$ with the same end-edges and the same $\ell$-ends as $A$. Note that in each step, we replace a 3-edge sub-path of $A$ by a 4-edge path with the same \nak{end-edges}, so the resulting path remains properly colored. \end{proof}

\subsection{Covering by long paths}
       We emphasize that for the proof of Lemma~\ref{l:pathcover_ell-cycle} the argument from~\cite{hs10} goes through practically verbatim. All we have to do is to incorporate the coloring constraints into Propositions 19 and Lemma 20 in \cite{hs10}. We begin with a colored version of Proposition~19.

       For disjoint  subsets $V_1,  V_2,\dots, V_k\subset V$ of the vertex set a $k$-graph $\cal J$, a $(k,\ell)$-path $P = v_1v_2\ldots v_{s}$  is called \emph{canonical}
        if for each $j=i,\dots,s$,
        $$v_i\in (V_1\cup\cdots\cup V_\ell)\cup (V_{k-\ell+1}\cup\cdots\cup V_k)\quad\mbox{if and only if}\quad deg_P(v_i)=2.$$
        The following result is an analog of Claim~\ref{l:densepath}. Note that, trivially, $\Delta_{k-1}({\cal J}) \leq dm$ implies $\Delta_{\ell}({\cal J}) \leq dm^{k-\ell}$.

        \begin{claim} \label{l:densepath_ell-cycle}
            Let $(\cal J, \phi)$ be a colored $k$-partite $k$-graph with partition classes $W_1,\dots,W_k$, $|W_i|\le m$ for all $i\in[k]$, and at least $dm^k$ edges. If $\Delta_{\ell}({\cal J}_i) \leq dm^{k-\ell}/4$ for all $i$, then $\cal J$ contains a \pc canonical path on at least $dm/4$ vertices.
        \end{claim}

        \begin{proof} As in the proof of Proposition 19 in \cite{hs10}, by deleting iteratively some edges of $\cal J$, we construct a (non-empty) sub-hypergraph $\cal J'$ of $\cal J$ in which all $\ell$-element sets $L$ with exactly one vertex in each $W_i$, $i=1,\dots,\ell$, as well as, all $\ell$-element sets $L$ with exactly one vertex in each $W_i$, $i=k-\ell+1,\dots,k$ satisfy: $deg_{\cal J'}(L)=0$ or $deg_{\cal J'}(L)\ge dm^{k-\ell}/2$.

            Let $P = v_1v_2\ldots v_{s}$ be the longest \pc canonical path in $\cal J'$. The set $L=\{v_{s-\ell+1},\dots,v_{s}\}$ has, clearly, a nonzero degree in $\cJ'$, and so it is contained in at least $dm^{k-\ell}/2$ edges of $\cal J'$. Moreover, by our assumption,  at most $dm^{k-\ell}/4$ of them have the same color as the last edge of $P$. Hence, there are at least  $dm^{k-\ell}/2-dm^{k-\ell}/4=dm^{k-\ell}/4$ edges in $\cal J'$ which contain $L$ and have a different color than the last edge of $P$. On the other hand, by the maximality of $P$, each one of these edges must intersect $V(P)\setminus L$ and so there cannot be more than $sm^{k-\ell+1}$ such edges. Consequently,
            $dm^{k-\ell}/4\le sm^{k-\ell+1}$
            which implies that $s\ge dm/4$, as required.
            \end{proof}

            The outcome of the proof of Lemma 7 in \cite{hs10}, excluding the last two sentences, is summarized in the following lemma.
            \begin{lemma} \label{l:regtriples_ell}
            For sufficiently small $\eps:=\eps(\gamma)>0$, there exists an integer $T_0$  such that the following holds. Any $n$-vertex $k$-graph $\cH$ with $\delta_{k-1}(\cH)\ge(1/2(k-\ell)+\gamma)n$ contains a collection  $\cal C $ of at most $T_0$ vertex-disjoint $(\eps, \gamma/6)$-regular $k$-tuples $(U^j_1,\dots, U^j_{k-1},U^j_k)$, $1\le j\le |{\cal C}|\le T_0$, of sets of sizes $|U^j_1|=\cdots=|U^j_{k-1}|=(2k-2\ell-1)m$ and $|U^j_k|=(k-1)m$, for some $m$, which cover all but $\eps n$ vertices of $\cH$.  \qed
        \end{lemma}
\noindent In fact,  in \cite{hs10} the number of such collections is at most $\tfrac{2k-2\ell-1}{2k-2\ell}t$, where $t\le T_{14}$ and $T_{14}$ is  a constant delivered by the Weak Regularity Lemma (Lemma 14 in \cite{hs10}).

            As a final ingredient of the proof of Lemma \ref{l:pathcover_ell-cycle}, we now prove a colored modification of Lemma 20 in \cite{hs10}.
     \begin{lemma}   \label{l:regpaths_ell}
           For all $\beta >0$ and $d>0$ there exist $c', \eps>0$ and $q \in \mathbb{N}$ such that the following holds for sufficiently large $m$.
           Let $(\cal J, \phi)$ be a colored $k$-partite $k$-graph with the partition classes forming an $(\eps, d)$-regular $k$-tuple $(U_1,\dots, U_{k-1},U_k)$ with $|U_1|=\cdots=|U_{k-1}|=(2k-2\ell-1)m$ and $|U_k|=(k-1)m$. If $\Delta_{\ell}({\cal J}_i) \leq c'm^{k-\ell}$, then there is a family of at most $q$ vertex-disjoint, \pc paths which cover all but at most $\beta m$ vertices of~$\cal J$.
           \end{lemma}

           \proof A brief analysis of the proof o Lemma 20 in \cite{hs10} reveals that the only alteration is an application of our Claim \ref{l:densepath_ell-cycle} instead of their Proposition 19. To do so, we need to set $c'\le \eps(2k\eps)^{k-\ell}/4$. Indeed, then for the sets $W_i \subset U_i$ defined in the proof of Lemma 20 in \cite{hs10}, setting $m'=2k\eps m$ and assuming $\eps\le d/2$,
           we have $|W_i|=m'$, $i=1,\dots,k$, and
            $$e(W_1,\dots, W_k)\ge(d-\eps)(m')^k\ge \eps(m)'^k,$$ while
           $$\Delta_\ell({\cal J}[W_1,\dots, W_k])\le c'm^{k-\ell}=\frac{c'}{(2k\eps)^{k-\ell}}(m')^{k-\ell}\le\frac{\eps}4 (m')^{k-\ell}.$$
           Thus, by Claim \ref{l:densepath_ell-cycle} with $d:=\eps$, ${\cal J}[W_1,\dots, W_k]$ contains a \pc canonical path on at least $\eps m'/4$ vertices, which leads to a contradiction in the proof of Lemma 20 in \cite{hs10} (note that in \cite{hs10} the length of a path is measured by the number of edges rather than vertices). The rest of that proof carries on unchanged.
        \qed

\begin{proof}[Proof of Lemma~\ref{l:pathcover_ell-cycle}]
    For any $\gamma>0$ and $\delta>0$,  let $\eps_{\ref{l:regtriples_ell}}$ and $T_0$ be as in Lemma \ref{l:regtriples_ell} and let $\beta\le \delta k^2/(2T_0)$.
     Further, set $d=\gamma/6$ and let $\eps_{\ref{l:regpaths_ell}},c',q$ be as in Lemma \ref{l:regpaths_ell}.
Finally, set $\eps=\min\{\eps_{\ref{l:regtriples_ell}},\eps_{\ref{l:regpaths_ell}},\delta/2,\gamma/12\}$.   We are going to prove Lemma~\ref{l:pathcover_ell-cycle} with $c=c'/(4T_0k^2)^{k-\ell}$ and $Q=T_0q$.

    By Lemma \ref{l:regtriples_ell}, $\cH$ contains a collection  \nak{of vertex-disjoint $(\eps, \gamma/6)$-regular $k$-tuples ${\cal C} = \{ (U^j_1,\dots, U^j_{k-1},U_k^j) : 1\le j\le |{\cal C}| \}$}, such that $|{\cal C}|\leq T_0$, the sets $U_i^j$ cover all but $\eps n$ vertices of $\cH$ and satisfy $|U^j_1|=\cdots=|U^j_{k-1}|=(2k-2\ell-1)m$ and $|U^j_k|=(k-1)m$ for an integer $m$.  Note that
    $$(1-\eps)n\le |\cC|\left[(k-1)(2k-2\ell-1)+(k-1)\right]m\le 2T_0k^2m.$$
    Since $\eps\le 1/2$, the above estimate implies that $n\le (4T_0k^2)m$. On the other hand, we also have
    $$|\cC|\left[(k-1)(2k-2\ell-1)+(k-1)\right]m\le n,$$
    and so $m\le n/k^2$.

    We now verify the hypothesis of   Lemma~\ref{l:regpaths_ell}.
    Setting $\cJ^j=\cH[U^j_1,\dots,U^j_k]$, we have
    $$\Delta_\ell(\cJ_i^j)\le \Delta_\ell(\cH_i)\le cn^{k-\ell}\le c(4T_0k^2)^{k-\ell}m^{k-\ell}=c'm^{k-\ell}.$$
    Thus, by Lemma \ref{l:regpaths_ell}, for each $j$ there is a family $\cP^j$ of at most $q$ vertex-disjoint, \pc paths in $\cJ^j$ which cover all but at most $\beta m$ vertices of~$(U^j_1,\dots, U^j_{k-1},U_k\nak{^j})$.

    Consider the family $\bigcup_{j=1}^{|\cC|}\cP^j$. It consists of at most $|\cC|q\le T_0q=Q$ vertex-disjoint, \pc paths. Moreover, the number of vertices of $V$ not covered by these paths, by our estimates on $\beta$, $m$, and $\eps$, is at most
    $$T_0(\beta m)+\eps n\le \frac\delta2 k^2\frac n{k^2}+\frac\delta2 n=\delta n.$$
    This completes the proof of Lemma~\ref{l:pathcover_ell-cycle}.
    \end{proof}


\section{Loose Hamilton cycles with degree conditions}      \label{sec:loose}

This section is devoted to the proof of Theorem~\ref{loose_singles}. As all  our statements below are about loose paths and cycles, the attribute `loose' will sometimes be dropped. Recall that  given a $3$-graph $\cH$ and a coloring $\phi$ of its edges, $\cH_i=\{e\in\cH: \phi(e)=i\}$.
Throughout this section,
given a sufficiently small $\gamma \in (0, 1)$, $(\cH,\phi)$ is an $n$-vertex colored $3$-graph on vertex set $V$ with $n$ sufficiently large,
\begin{equation}\label{setup1}
    \delta_1(\cH)\ge(7/16+\gamma)n^2/2\quad\mbox{and}\quad\Delta_1(\cH_i)\le cn^2\quad\mbox{for all}\quad i\in\bN,
     \end{equation}
     where  $c>0$ is sufficiently small with respect to $\gamma$ and  some other constants introduced later.

	Our proof follows, again, the standard absorbing method. However, there is a serious obstacle related to the act of absorption which affects two of the three crucial lemmas.

	Due to the structure of loose paths, one cannot absorb vertices into $P$ one by one, but rather in pairs. Moreover, the proper-coloring constraint results in not every pair of vertices being absorbable. Thus,  we introduce an auxiliary graph $G$ on $V$ whose edges represent the pairs of vertices which can be absorbed into $P$.  Recall that the 1-ends of  loose path $P=v_1\cdots v_s$  are just $v_1$ and $v_s$.
	
    \begin{lemma}[Absorbing Lemma] \label{l:abspath_loose}
        For every $0<\lambda\le 10^{-14}$ and $c\le 30^{-9}$, if (\ref{setup1}) holds (even with $\gamma=0$), then
        there is  a \pc  loose path $A$ in $\cH$ of order at most $12\lambda n$ and a graph $G = G(\cH, \phi)$ on $V$ with $\delta(G) \geq \frac 34 n$ such that the following holds.

        For every $U\subset V$ with $|U| \leq 2\lambda^2 n$, if $G[U]$ has a perfect matching, then there is a \pc loose path $P_U$ in $\cH$ with  $V(P_U)=V(A) \cup U$  with the same
        end-edges and $\ell$-ends as $A$. (We will say that $A$ can absorb $U$.)
    \end{lemma}

Part (b) of the next lemma ensures that the vertices left outside the long cycle at the end of the proof  induce a perfect matching in the graph $G$ guaranteed by Lemma~\ref{l:abspath_loose}.

    \begin{lemma}[Reservoir Lemma] \label{l:reservoir_loose}
        Let $G$ be a graph on $V$ with $\delta(G) \geq  0.71 n$.
        For every  $\rho >0$ and $c \leq \frac{\rho^2}{4\cdot 10^5}$, if (\ref{setup1}) holds (even with $\gamma=0$), then there is a set $R \subset V$ with $ |R| \leq \rho n$, which has the following properties:
        \begin{enumerate}
            \item For any two vertices $x,y\notin R$, any two colors $c_x, c_y$, and any subset $R'\subset R$, $|R'|\le |R|/100$, there exist $v_1,v_2,v_3,v_4,v_5\in R\setminus R'$ such that $xv_1v_2v_3v_4v_5y$ is a \pc  loose path in $\cH$ with $\phi(xv_1v_2)\neq c_x$ and $\phi(v_4v_5y)\neq c_y$.
			\item If $U\subset V$, $|U|$ even, $|R\setminus U| \leq |R| / 100$, and $|U \setminus R| \leq |R| /100$, then $G[U]$ has a perfect matching.

        \end{enumerate}
    \end{lemma}

Only the covering lemma is not affected by the above-mentioned problem with absorption.

     \begin{lemma}[Covering Lemma] \label{l:pathcover_loose}
     For every $\gamma>0$ and $\delta>0$, let $ c=c(\gamma,\delta)>0$ be sufficiently small and $q=q(\gamma,\delta) $ be sufficiently large. If (\ref{setup1}) holds, then there is a family $\cP$ of at most $q$ vertex-disjoint \pc  loose paths in $(\cH,\phi)$ covering all but at most $\delta n$ vertices of~$\cH$.
    \end{lemma}

    \begin{proof}[Proof of Theorem~\ref{loose_singles}]
    Set
    $$\lambda = \min\left(10^{-14},\gamma/25\right),\quad\rho=\lambda^2,\quad\delta=\lambda^2/200,\quad q=q_{\ref{l:pathcover_loose}}(\gamma/2,\delta),$$
     and
     $$c = \min\left(30^{-9},\rho^2/(5\cdot 10^5),\frac12c_{\ref{l:pathcover_loose}}(\gamma/2,\delta)\right),$$
      where the subscript $_{\ref{l:pathcover_loose}}$ means that the constant comes from Lemma \ref{l:pathcover_loose}.

      Apply Lemma~\ref{l:abspath_loose} to find a graph $G$ with $\delta(G)\ge\tfrac34n$ and an absorbing path $A$ with $|V(A)|\le12\lambda n$ which can absorb any set $U$ of up to $2\lambda^2n$ vertices such that $G[U]$ contains a perfect matching.

    Next, let $R\subset V \setminus V(A)$ be the set given by Lemma~\ref{l:reservoir_loose} applied to $\cH'=\cH-V(A)$, $G'=G-V(A)$ and to the coloring $\phi'$ induced in $\cH'$ by $\phi$. Such an application is feasible, because, setting $n'=|V(\cH')|=n-| V(A)|$, we have
    $$\delta_1(\cH')\geq \delta_1(\cH)-|V(A)|n\ge\left(\frac{7}{16} + \gamma \right) \frac{n^2}{2}-12\lambda n^2 \geq \frac{7}{32} (n')^2,$$
    $$\Delta_1(\cH'_i)\le\Delta_1(\cH_i)\le cn^2\le\frac{\rho^2}{5\cdot 10^5(1-10\lambda)^2}(n')^2\leq \frac{\rho^2}{4\cdot 10^5}(n')^2,$$
    and
    $$\delta_1(G')\ge\delta(G)-|V(A)|\ge\frac34n-12\lambda n\ge 0.71n\ge 0.71n'.$$

        Further, let $\cH''=\cH'-R$  and let $\phi''$ be the coloring induced in $\cH''$ by $\phi$.
        Set $n'':=|V(\cH'')|$. Since, for $ \nak{ \lambda \leq \tfrac{\gamma}{25}}$, we have  $\lambda^2+12\lambda\le\gamma/2$,
        $$\delta_{1}(\cH'') \geq \delta_{1}(\cH)-(|V(A)|+|R|)n\ge\left( \frac {7}{16} + \frac \gamma 2 \right)(n'')^2/2$$
        and
        $$\Delta_1(\cH''_i)\le\Delta_1(\cH_i)\le cn^2\le\frac{c_{\ref{l:pathcover_loose}}}{2(1-\lambda/2)^2}(n')^2\leq c_{\ref{l:pathcover_loose}}(n'')^2.$$

        Hence, we can apply Lemma~\ref{l:pathcover_loose} to the colored hypergraph $(\cH'', \phi'')$ with $\gamma:=\gamma/2$ to obtain a family $\cal P$ of at most $q$ \pc paths which cover all but at most $ \lambda^2 n''/200$ vertices of $\cH''$. Let $W$ be the set of vertices of $\cH''$ not covered by any of the paths in $\cP$. Include the absorbing path $A$ into $\cP$ to get a family of paths $\cP^A =\cP\cup\{A\}$.

     To connect the paths from $\cP^A$ into one cycle $C$, we  apply Proposition \ref{l:metatheorem} with $Q=R$ and $m=|{\cP^A}|=q+1$. To verify Statement I therein, we invoke Lemma \ref{l:reservoir} with $R'=Q'$, so we set  $g=5$. For $n$ large enough, we have
$$|Q'|\le mg\le 5(q+1)\le\frac1{100}|R|.$$
 By Lemma \ref{l:reservoir}(i), applied to one $1$-end $x$ of $P_1$ and one $1$-end $y$ of  $P_2$, we see that Statement I of  Proposition \ref{l:metatheorem} holds  and thus Statement II follows.

        Let $U$ be the set of vertices of $V$ not covered by $C$. Since $|V(C)|$ is even, so is $|U|$. We have
        \[
        |U \setminus R| \le|W|\le \frac{\lambda^2 n''}{200} \leq \frac{|R|}{100}  ,\,\, \text{ and }\,\, |R \setminus U| \leq |R'| \le \frac{|R|}{100} . \
        \]
        Thus, by Lemma~\ref{l:reservoir_loose}(ii), the graph $G[U]$ has a perfect matching. Moreover, since
        $$|U| \le|W|+|R|\le\delta n+\rho n< 2\lambda^2 n,$$ $U$ can be absorbed into $A$ (by replacing $A$ with a \pc path $P_U$), and thus into $C$, forming  a Hamilton cycle in $\cH$. Since $P_U$ and $A$ have the same end-edges, the obtained Hamilton cycle remains properly colored.
    \end{proof}


\subsection{Connecting and Reservoir Lemmas}\label{con-res-loose}

    The following lemma states that if $x$ and $y$ are vertices already present in edges of given colors $c_x$ and $c_y$, respectively, then one can  connect them via a short \pc  path which avoids a given relatively small set of vertices without creating a color conflict. In fact, although we do not need it here, we show that there are many such paths.

\begin{lemma}[Connecting Lemma]\label{l:connecting_loose} Let a colored 3-graph $(\cH,\phi)$ be given with
$\delta_1(\cH)\ge n^2/5$ and $\Delta_1(\cH_i)\le 10^{-5}n^2$. Further, let  $Z\subset V(\cH)$, $|Z|\leq 0.01 n$. Then, for any pair of vertices $x,y\in V(\cH)\setminus Z$, and any two colors $c_x,c_y$, there exist at least $4(n/10)^5$ \pc paths  $xv_1v_2v_3v_4v_5y$, where $v_1,v_2,v_3,v_4,v_5\not\in Z$, such that, in addition, $\phi(xv_1v_2)\neq c_x$ and $\phi(v_2v_3y) \neq c_y$.
\end{lemma}

\begin{proof} We will first estimate from below the number of paths $xv_1v_2v_3v_4v_5y$ with no regard to coloring. Then we will subtract an upper bound on the number of those among them which have a color conflict. Given a vertex $u\in V(\cH)$, \nak{we define its \emph{neighborhood} in $\cH$ as $N_{\cH}(u) = \{S \subset V(H): S \cup \{ v\} \in \cH \}$} and denote  $L_u=N_{\cH}(u)$.  Note that under our assumptions, $L_u$ is a graph with $|L_u|\ge n^2/5$.

 Choose $\{v_4,v_5\}\in L_{y}$, $v_4,v_5\neq x$, arbitrarily. There are at least $n^2/5-n$ ways to do so. Turning, for a moment, to the other end of the to-be-path, let $X=\{v\in V: \deg_{L_x}(v)\le n/20\}$ and set $|X|=\tau n$. Then
$$\frac15n^2-n\le|L_x|\le|X|n/20+\binom{|V\setminus X|}2\le\left(\frac1{10}\tau+(1-\tau)^2\right)n^2/2,$$
which implies (since $n$ is arbitrarily large) that
$$f(\tau):=\frac1{10}\tau+(1-\tau)^2\ge\frac25.$$
Note that $f$ has the unique minimum at $\tfrac{19}{20}$, while $f(1) = 1/10 < 2/5$ and $f(2/5)=2/5$. It follows that $\tau\le 2/5$. As the same is true for $L_{v_4}$, we infer that there are at least $n/5$ vertices $v\in V$ with both $\deg_{L_x}(v)\ge n/20$ and $\deg_{L_{v_4}}(v)\ge n/20$. Thus, avoiding $Z\cup\{x,y,v_4,v_5\}$, we have at least
$$(n/5-|Z|-4)(n/20-|Z|-4)(n/20-|Z|-5)\ge \frac3{10^4}n^3$$ triples $(v_1,v_2,v_3)$ with $v_1,v_2,v_3\not\in Z$, $\{v_1,v_2\}\in L_x$ and $\{v_2,v_3\}\in L_{v_4}$. Consequently, factoring in the number of ordered pairs $(v_4,v_5)$, there are at least
$$2\left( \frac{n^2}{5} -n \right)\frac3{10^4}n^3\ge\frac8{10^{5}}n^5$$ paths $xv_1v_2v_3v_4v_5y$ with $v_1,v_2,v_3,v_4,v_5\not\in Z$.

Now, we count the paths with color conflicts. On top of the forbidden conflicts $\phi(xv_1v_2)= c_x$ and $\phi(v_2v_3y) =c_y$,
there are two more corresponding to the two pairs of edges intersecting, respectively, at $v_2$ and at $v_4$. Using the assumption $\Delta_1(\cH_i)\le 10^{-5}n^2$, each one of them disqualifies at most  $(n/10)^5$. Altogether,  the number of required paths is at least
$$8 \left(\frac{n}{10} \right)^5-4\left(\frac{n}{10} \right)^5 \ge 4 \left(\frac{n}{10} \right)^5.$$
\end{proof}



\medskip

\begin{proof}[Proof of Lemma~\ref{l:reservoir_loose}]
        Let $p = \frac23\rho$.
         Applying Proposition \ref{res} with $U_v=N_G(v)$, $G_v=N_{\cH}(v)$, $\alpha_v=0.71$, $\beta_v=\tfrac7{16}$, for all $v\in V$, we obtain a subset $R\subset V$ such that $\tfrac12\rho n\le |R|\le\rho n$, and, for all $v\in V$,

\begin{enumerate}

\item[(1)]  $|N_G(v)\cap R|\ge (0.71-2n^{-1/3})|R|\ge\tfrac23|R|$, and

\item[(2)]  $|N_{\cH}(v)[R]|\ge (\tfrac7{16}-3n^{-1/3})\binom{|R|}2\ge \tfrac{13}{32}\binom{|R|}2$.
\end{enumerate}

 We claim that for any set $R$ satisfying properties (1) and (2) above, conditions (i) and (ii) of Lemma ~\ref{l:reservoir_loose} hold.

 To prove (i) consider a pair of vertices $x,y\notin R$, two colors $c_x, c_y$, and a subset $R'\subset R$ with $|R'|\leq |R|/100$. Let $\cR = \cH[R\cup\{x,y\}]$ and set $r = |V(\cR)| = |R| + 2$. Since $r\ge \rho n/2$ and $n$ is sufficiently large, by (2),
        \[
        \delta_1(\cR) \geq \frac{13}{32}{|R| \choose 2}\ge\frac{r^2}5,
        \]
        and, by our assumption on $c$, for each $i$,
        \[
        \Delta_1(\cR_i) \leq cn^2 \leq \frac{4c}{\rho^2} |R|^2\le 10^{-5} r^2 .
        \]
         Thus, we are in position to apply Lemma~\ref{l:connecting_loose} with $\cH := \cR$, $\phi := \phi|_{\cR}$, $Z := R'$, and $c := 4c/\rho^2$, obtaining  the desired  $v_1,v_2,v_3,v_4,v_5\in R\setminus R'$.

To prove (ii), it is enough to show that $\delta(G[U])\ge\frac 12 |U|$, as this, recalling that $|U|$ is even, implies the existence of a perfect matching in $G[U]$. To this end, note that, since $|U\setminus R|\le|R|/100$, we have $|U|\le |R|+|R|/100$, or equivalently, $|R|\ge\tfrac{100}{101}|U|$. Thus, for any $u \in U$, using (1) and the inequality $|R \setminus U|\le|R|/100$,
$$|N_G(u) \cap U| \geq |N_G(u) \cap R| - |R \setminus U| \geq \frac 23 |R|- \frac{1}{100}|R|=\frac{197}{300}|R|\ge\frac{197}{303}|U| > \frac{1}{2}|U|,$$
which completes the proof of the lemma.
\end{proof}


\subsection{Absorbing path}

In this subsection we prove  Lemma~\ref{l:abspath_loose}. As usual, the absorbing path will be built from small pieces called \emph{absorbers}.
For a pair of distinct vertices $x,y\in V$, an $(x,y)$-\emph{absorber} is a 7-tuple $\seq{v}=(v_1,v_2,\ldots,v_7)$ of vertices of $\cH$ such that:
\begin{enumerate}
    \item $v_1v_2v_3, v_3v_4v_5, v_5v_6v_7 \in \cH$,
    \item $v_2xv_4, v_4yv_6 \in \cH$.
\end{enumerate}
	In other words, both $\seq{v}$ and $\seq{v}_{x,y} := (v_1, v_3, v_2, x, v_4, y, v_6, v_5, v_7)$ induce loose paths in $\cH$ in that ordering. If both $\seq{v}$ and $\seq{v}_{x,y}$ are properly colored in $\cH$, then $\seq{v}$ is called a \emph{properly colored $(x,y)$-absorber}. Note that $\seq{v}$ and $\seq{v}_{x,y}$ have the same end-edges and 1-ends, though the order of vertices in the end-edges does change.
%

In (\cite{bhs13}, Proposition 8) it was shown that under some milder degree assumptions, for any pair of vertices $x,y\in V(\cH)$, the number of $(x,y)$-absorbers is $\Theta(n^7)$.

\begin{prop}[\cite{bhs13}]
For every $\xi\in(0,3/8)$ there exists $n_0$ such that the following holds. Suppose that $\cH$ is a 3-graph on $n > n_0$ vertices with $\delta_1(\cH) \geq (\frac58 + \xi)^2{n\choose 2}$. Then for every pair of vertices $x,y\in V(\cH)$ the number of $(x,y)$-absorbers is at least $(\xi n)^7/8$.
\end{prop}

Under our stronger assumption on $\delta_1(\cH)$, we may take $\xi=1/30$ and deduce the following.

\begin{cor}\label{cor:many_absorbers}
There exists $n_0$ such that if $\cH$ be a 3-graph on $n > n_0$ vertices with $\delta_1(\cH) \geq \frac7{16}{n\choose 2}$, then for every pair of vertices $x,y\in V(\cH)$ the number of $(x,y)$-absorbers is at least $\tfrac18 (n/30)^7$.
\end{cor}

Consider a pair of vertices $x,y\in V(\cH)$. Although there are many $(x,y)$-absorbers, there is no guarantee that any of them are properly colored. Indeed, it may be the case that for every choice of $v_2, v_4, v_6\in V(\cH)$, we have $\phi(v_2xv_4) = \phi(v_4yv_6)$. The following proposition states that for every vertex $x$ there are many vertices $y$ such that the number of \pc $(x,y)$-absorbers is of the order of magnitude as large as possible. A pair of distinct vertices $x,y\in V(\cH)$ is called \emph{pc-absorbable} if there exists at least $\tfrac1{16} (n/30)^7$ \pc $(x,y)$-absorbers.

\begin{prop}\label{prop:many_absorbable_pairs}
Let $c\le 30^{-9}$. For every $x\in V$, there are at least $\frac34 n$ vertices $y\in V$ such that
 the pair $(x,y)$ is pc-absorbable.
\end{prop}

\begin{proof}
Given $x\in V$   it is enough to show that for each $Y\subset V$, $|Y|=n/4$, there is a vertex $y\in Y$ such that the pair $(x,y)$ is pc-absorbable.

Set $\beta=\tfrac1{16} (n/30)^7$ for convenience. By Corollary~\ref{cor:many_absorbers}, the number of pairs $(y,\seq{v})$, where $y\in Y$ and $\seq{v}=(v_1,\ldots,v_7)$ is an $(x,y)$-absorber is at least $\frac12 \beta n^8$.
We wish to count for how many of these pairs $\seq v$ is not a \pc  $(x, y)$-absorber, i.e.~at least two intersecting edges in $\seq v$ or in ${\seq v}_{x, y}$ have the same color. As there are in total five intersecting pairs among the edges of the two paths (three in $\seq v$ and two in ${\seq v}_{x, y}$), this number can be crudely bounded from above by $5cn^8$.

Indeed, fixing two intersecting edges, say, $v_1v_2v_3$ and $v_3v_4v_5$, the number of pairs $(y,\seq{v})$ in question can be bounded from above by first bounding crudely the number of choices of $y,v_1,v_2,v_3,v_6,v_7$ by $n^6$ and then, setting $i:=\phi(v_1v_2v_3)$, bounding the number of choices of $v_4,v_5$ by $\Delta(\cH_i)\le cn^2$, as $v_3v_4v_5$ must be of the same color as $v_1v_2v_3$. If one of the intersecting edges involves $x$, as in the pair $v_2xv_4,\;v_4yv_5$, then the bound $cn^2$ is applied to the other edge in the pair (in the given example, we would use this bound to the number of choices of $y,v_5$).

Consequently, the number of pairs $(y,\seq{v})$, where $y\in Y$ and $\seq{v}$ is a properly colored $(x,y)$-absorber is at least
\[
\left(\frac12 \beta - 5c\right) n^8.
\]
By averaging, there exists $y\in Y$ for which the number of \pc $(x,y)$-absorbers is at least
$$(2\beta - 20c)n^7 \geq \beta n^7=\frac1{16} (n/30)^7,$$
where we also used the inequality $20c<\beta$.
Thus, the pair $(x,y)$ is pc-absorbable.
\end{proof}

\begin{proof}[Proof of Lemma~\ref{l:abspath_loose}]
By definition, for any pc-absorbable pair $(x, y)$, the family $\A_{(x, y)}$ of properly colored $(x,y)$-absorbers has size
 $$|\A_{(x, y)}|\ge 2^{-4} (n/30)^7  > 4 \cdot 7^2 \lambda n,$$
  where the last inequality follows by our assumption on $\lambda$. The elements of $\bigcup_{(x,y)}\cA{(x,y)}$ will be called \emph{absorbers}.
  Applying Proposition~\ref{l:absorbersampling} to the families $\A_{(x,y)}$ with $t=7$ and $\alpha = \lambda$, one can find a family $\F$ of absorbers satisfying $|\F|\le \lambda n$ and $|\F \cap \A_{(x, y)}| \geq \tfrac{49}4\lambda^2 n$ for all pc-absorbable pairs $(x, y)$.

 To connect the paths from $\F$ into one path $A$, we  apply Proposition \ref{l:metatheorem} with $Q=V\setminus\bigcup_{F\in\F}V(F)$ and $m=|\F|\le\lambda n$. To verify Statement I therein with $g=5$, let $Q'$, $P_1$ and $P_2$ in $(\cH - R, \phi)$ be as in the statement. We  invoke Lemma \ref{l:con} with $Z=Q'\cup \bigcup_{F\in\F}V(F)$, so we set  $b=5$. Note that, for $n$ large enough,
    \[
    |Z|\le |Q'|+7|\F|\le mg+ 7\lambda n\le12\lambda n\le0.01n.
    \]
    By Lemma \ref{l:con}, applied to one $1$-end $x$ of $P_1$ and one $1$-end $y$ of $P_2$, we see that Statement~I of  Proposition \ref{l:metatheorem} holds  and thus Statement II follows. Note that the obtained path~$A$ has length
    $$|V(A)| \le 12 |\F| \le 12\lambda n,$$
    as required.

Let $G$ be an auxiliary graph on vertex set $V$ whose edges correspond to pc-absorbable pairs. Observe that there is no ambiguity in this definition, since, by  reversing  the 7-tuples, $(x, y)$ is a pc-absorbable pair if and only if $(y, x)$ is. Proposition~\ref{prop:many_absorbable_pairs} implies that $\delta(G) \geq \frac 34 n $. Consider a set $U\subset V$ of size $|U|\le 2\lambda^2n$ which induces a matching $M$ in~$G$. For each edge $xy$ of $M$, the path $A$ contains
$$|\A_{(x, y)}|\ge \frac{49}4\lambda^2 n\ge\lambda^2n$$ \pc $(x,y)$-absorbers. Hence, one can absorb the pairs of $M$ greedily to obtain a \pc path $P_U$ on the vertex set $V(A) \cup U$. Note that, by construction of  absorbers, in each step the 1-ends and end-edges of the resulting path stay intact. Thus, the final path $P_U$ has the same 1-ends and end-edges as the path $A$.
\end{proof}
%



\subsection{Covering by long paths}
        The proof of Lemma~\ref{l:pathcover_loose} follows closely the proof of Lemma 10 in~\cite{bhs13} which, in turn, is based partly on Lemma 20 from~\cite{hs10}. Since we already have a colored version of that lemma, namely Lemma \ref{l:regpaths_ell}  in Section~\ref{sec:ell-cycle}, all we need is a decomposition result analogous to Lemma \ref{l:regtriples_ell}. Such a result can be deduced from the results in \cite{bhs13}.
        It is based on the Weak Hypergraph Regularity Lemma (Proposition 15 in \cite{bhs13})  and an extremal result developed in~\cite{bhs13} (Lemma 11 therein);  this deduction is shown in the proof of Lemma 10 in \cite{bhs13}. (In our version, we apply Lemma 11 with $\alpha:=\eps$.) Notice that this statement makes no reference to  edge coloring.

        \begin{lemma}[\cite{bhs13}] \label{l:regtriples_loose}
             For a sufficiently small $\eps=\eps(\gamma) >0$, there exists an  integer $T_0$ such that the following holds. Any $n$-vertex $3$-graph $\cH$ with $\delta_1(\cH) \geq \left( \frac {7}{16} + \gamma \right) \binom n2$ contains a collection $\cC $ of at most $T_0$ vertex-disjoint $(\eps, \gamma/3)$-regular triples $(U_1^j,U_2^j,U_3^j)$, $j=1,\dots,|{\cC}| \leq T_0$, which cover all but $\eps n$ vertices of $\cH$. Moreover,  for some $m$, half of the triples satisfy the identity $|U_1^j|=|U_2^j|=\tfrac32|U_3^j|=3m$, while for the other half $|U_1^j|=|U_2^j|=\tfrac32|U_3^j|=6m$. \qed
        \end{lemma}

        \begin{proof}[Proof of Lemma~\ref{l:pathcover_loose}] In this proof, which is very similar to that of Lemma~\ref{l:pathcover_ell-cycle}, we utilize Lemma \ref{l:regtriples_loose} combined with Lemma \ref{l:regpaths_ell} for $k=3$ and $\ell=1$.
         For any $\gamma>0$ and $\delta>0$,  let $\eps_{\ref{l:regtriples_loose}}$ and $T_0$ be as in Lemma \ref{l:regtriples_loose} and let $\beta\le 6\delta$.
     Further, set $d=\gamma/3$ and let $\eps_{\ref{l:regpaths_ell}},c',q$ be as in Lemma \ref{l:regpaths_ell} with $k=3$ and $\ell=1$.
Finally, set $\eps=\min\{\eps_{\ref{l:regtriples_loose}},\eps_{\ref{l:regpaths_ell}},\delta/2,\gamma/6\}$.   We are going to prove Lemma~\ref{l:pathcover_ell-cycle} with $c=c'/(24T_0)^{2}$ and $Q=T_0q$.

    By Lemma \ref{l:regtriples_loose}, $\cH$  contains a collection  $\cal C $ of size $|{\cal C}|\le T_0$ of vertex-disjoint $(\eps, \gamma/3)$-regular triples of prescribed sizes which cover all but $\eps n$ vertices of $\cH$. Note that
    $$(1-\eps)n\le 12|\cC|m\le 12T_0m.$$
    Since $\eps\le 1/2$, the above estimate implies that $n\le 24T_0m$. On the other hand, we also have
    $12|\cC|m\le n.$

    In order to apply  Lemma \ref{l:regpaths_ell}, we need to verify its assumption (with $k=3$ and $\ell=1$).
    Setting, $\cJ^j=\cH[U^j_1,U^j_2,U^j_3]$, we have
    $$\Delta_1(\cJ_i^j)\le \Delta_1(\cH_i)\le cn^{2}\le c(24T_0)^{2}m^{2}=c'm^{2}.$$
    Thus, by Lemma \ref{l:regpaths_ell}, for each $j$ there is a family $\cP^j$ of at most $q$ vertex-disjoint, \pc paths in $\cJ^j$ which cover all but at most $\beta m$ vertices of~$(U^j_1,U^j_2,U^j_3)$.

    Consider the family $\bigcup_{j=1}^{|\cC|}\cP^j$. It consists of at most $|\cC|q\le T_0q=Q$ vertex-disjoint, \pc paths. Moreover, the number of vertices of $V$ not covered by these paths, by our estimates on $\beta$, $m$, and $\eps$, is at most
    $$|{\cal C}|(\beta m)+\eps n\le \beta \frac n{12}+\frac\delta2 n \leq \delta n.$$
    This completes the proof of Lemma~\ref{l:pathcover_loose}.

        \end{proof}


\section{Concluding remarks}
	We point out an extension and some directions for further research.
		\paragraph{Compatibility systems:} Theorems~\ref{tight}-\ref{loose_singles} actually hold in the setting of `compatibility systems' which generalise edge colorings. Compatible cycles were considered for instance in~\cite{kls17,glxz18}.
		\paragraph{Dirac hypergraphs:} Do our theorems hold for Dirac hypergraphs, that is, with the constant $\gamma =0$? Analogous results are presently only known for graphs~\cite{kls17,gj18}.
		\paragraph{The rainbow setting:}
		    It is likely that the rainbow analogues of Theorems~\ref{tight}-\ref{loose_singles}
hold, where	an additional assumption on the number of occurrences of each colour is needed. This was already asked by~\cite{ckpy18}, but let us spell out the conjecture for tight cycles. For loose cycles, the conjecture may even hold without the local-boundedness assumption, as in~\cite{dfr12,ksv17}.
        \begin{conj}\label{c:tight} For every  $k\ge3$ and $\gamma>0$ there exist $c>0$ and $n_0>0$ such that if  $(\cH,\phi)$ is an $n$-vertex  colored $k$-graph with $n \geq n_0$, $\delta_{k-1}(\cH)\ge(1/2+\gamma)n$,  $\Delta_{0}(\cH_i)\le cn^{k-1}$  and $\Delta_{k-1}(\cH_i)\le cn$ for every~$i\in\bN$, then $(\cH,\phi)$ contains a properly colored tight Hamilton cycle $C_n^{(k)}(k-1)$.
\end{conj}
			Existing rainbow results for incomplete hypergraphs~\cite{ckpy18} are achieved using the method of switchings, which is suitable for embedding $H$-factors because it exploits their symmetry. In other words, there are many ways to locally modify an  $H$ factor and obtain another $H$-factor. It is likely that new ideas are needed for connected spanning hypergraphs.





\end{document}